\documentclass{article}

\usepackage{graphicx}
\usepackage{amsmath,amssymb}
\usepackage{amsthm}
\usepackage{mathtools}
\usepackage{bbm}
\usepackage{float}

\DeclareMathOperator{\Tr}{Tr}
\newcommand{\norm}[1]{\left\lVert#1\right\rVert}

\renewenvironment{abstract}{%
  \small \textbf{\abstractname}.}{\newline}

\providecommand{\keywords}[1]
{\indent \small	\textbf{Key words.} #1 \newline}

\providecommand{\subjclass}[1]
{\indent \small \textbf{AMS subject classifications.} #1}

\newtheorem{theorem}{Theorem}[section]
\newtheorem{lemma}[theorem]{Lemma}
\newtheorem{remark}[theorem]{Remark}
\allowdisplaybreaks

\title{An approximate maximum likelihood estimator of drift parameters in a multidimensional diffusion model\footnote{This work has been fully supported by Croatian Science Foundation under the project IP-2020-02-9559.}}
\author{Miljenko Huzak\footnote{Department of Mathematics, Faculty of Science, University of Zagreb, 10 000 Zagreb, Croatia  (miljenko.huzak@math.hr)}, Snje\v zana Lubura Strunjak\footnote{Department of Mathematics, Faculty of Science, University of Zagreb, 10 000 Zagreb, Croatia (snjezana.lubura.strunjak@math.hr)}, Andreja Vlahek \v Strok\footnote{Faculty of Chemical Engineering and Technology, University of Zagreb, 10 000 Zagreb, Croatia (avlahek@fkit.hr)}}

\begin{document}

\maketitle

\begin{abstract}
For a fixed $T$ and $k \geq 2$, a $k$-dimensional vector stochastic differential equation $dX_t=\mu(X_t, \theta)\,dt+\nu(X_t)\,dW_t,$ is studied over a time interval $[0,T]$. Vector of drift parameters $\theta$ is unknown. The dependence in $\theta$ is in general nonlinear. We prove that the difference between approximate maximum likelihood estimator of the drift parameter $\overline{\theta}_n\equiv \overline{\theta}_{n,T}$ obtained from discrete observations $(X_{i\Delta_n}, 0 \leq i \leq n)$ and maximum likelihood estimator $\hat{\theta}\equiv \hat{\theta}_T$ obtained from continuous observations $(X_t, 0\leq t\leq T)$, when $\Delta_n=T/n$ tends to zero, converges stably in law to the mixed normal random vector with covariance matrix that depends on $\hat{\theta}$ and on path $(X_t, 0 \leq t\leq T)$. The uniform ellipticity of diffusion matrix $S(x)=\nu(x)\nu(x)^T$ emerges as the main assumption on the diffusion coefficient function.
\end{abstract}\\
\\

\section{Introduction}
Let $X=(X_t, \, t \geq 0)$ be a multidimensional diffusion with values in an open and convex state space $E\subseteq\mathbb{R}^k$ which satisfies It\^o stochastic differential equation
\begin{align}\label{sde}
dX_t=\mu(X_t, \theta)\,dt+\nu(X_t)
\,dW_t, \quad X_0=x_0,\quad t \geq 0,
\end{align}
where $W=(W_t, \, t \geq 0)$ is a $k$-dimensional Brownian motion and $x_0\in E$ is a known nonrandom initial state of $X$. Componentwise, it is a system of differential equations of the following form
\begin{align*}
X_t^{i}=x_{0}^{i} + \int_{0}^{t}\mu_{i}(X_s,\theta)\,ds + \sum_{j=1}^k \int_{0}^{t} 
\nu_{ij}(X_s)\,dW_s^j, \quad \text{for } i=1,2, \dots, k.
\end{align*}
For $d\geq 1$, let $\theta \in \Theta \subseteq \mathbb{R}^d$ be 
a drift parameter, or a vector of drift parameters if $d\geq 2$. Let us denote by $\theta_0$ its true value, and 
$\mathbb{P}\equiv \mathbb{P}_{\theta_0}$. Function $\mu(x,\theta)$ is called drift function and 
$S(x)=\nu(x)\nu(x)^T$ diffusion matrix. Diffusion coefficient function $\nu(x)$ may contain parameters that we consider known.\par
Let $T>0$ be fixed. For a fixed $n \in \mathbb{N}$, let $0 \eqqcolon t_0 < t_1 < \dots < t_n \coloneqq T$ be such a subdivision of the time interval $[0,T]$ that for $i=1, \dots, n$,  $t_i=i\Delta_n$ and $\Delta_n=T/n$. For a given discrete observation $(X_{t_i},\, 0 \leq i \leq n)$ of $X$ over time interval $\left[0, T\right]$,
our goal is to estimate the vector of parameters $\theta$ belonging to $\Theta$. We assume that $\Theta$ is an open, relatively compact and convex set in $\mathbb{R}^d$. In this paper, we use the Euler approximation of Riemann and Itô integrals in \eqref{sde} to obtain the set of difference equations whose solution is an approximation of the diffusion $X$. A similar estimation technique has been used in the case of a one-dimensional diffusion (\cite{huzak2018, lubhuz}). In this one-dimensional case it has been proved the existence of a sequence of so-called approximate maximum likelihood estimator (AMLE) 
which converges in probability to the maximum likelihood estimator with respect to a continuous observation (MLE) 
when $n$ goes to infinity \cite{huzak2001}. This result was obtained under smoothness assumptions of log-likelihood functions of the initial process and the approximation process, and under conditions that these  log-likelihood functions as well as their first and second derivatives were in some sense close. The same result can also be stated in a multidimensional framework.\par
In this paper we consider multidimensional diffusion ($k\geq 2$). We aim to prove that the difference between AMLE based on a discrete observation $\overline{\theta}_n$ and MLE based on continuous observations $\overline{\theta}_n$ over $[0,T]$ scaled by $\sqrt{\Delta_n}$ is asymptotically mixed normal when $\Delta_n$ tends to zero. The covariance matrix of the resulting random vector depends on the MLE and the path $(X_t, t \in [0, T])$.\par
The same result has been proved in \cite{lubhuz} for one-dimensional diffusion. In the multidimensional domain, the idea of the proof is similar, but the techniques are different. Advanced knowledge of linear algebra simplifies the notation and analysis. The components of $k$-dimensional Brownian motion are independent by definition, which often comes into play in the calculations in the proofs. There is also a new regularity condition for the diffusion matrix. It is a so-called uniform ellipticity condition (\cite{alfonsi,lee,gobet2001,gobet2002}). This condition actually means that the noise in the system is non-degenerate in any dimension and the process is somehow controlled (\cite{clairon}, \cite{gobet2005}). Most of the main results are first proved for a compact state space $E$ and then for an open $E \subseteq \mathbb{R}^k$.\par
The paper is organized as follows. In the next section we introduce the notation and give necessary definitions and known auxiliary results. Section 3 describes the estimation method. The assumptions and main theorems are stated in Section 4. The proofs of the main results are in Section 5, and an example and simulations are presented in Section 6. The Appendix consists of proofs of the lemmas.

\section{Preliminaries}
For two matrices $A$ and $B$ of the same dimension $m \times n$, the \textit{Hadamard product} $A \circ B$ is the matrix of the same dimension, with elements given by $\left(A \circ B\right)_{ij}=A_{ij}B_{ij}$ \cite{gentle}. We denote by $\left<\cdot | \cdot\right>$ the scalar product in Euclidean space $\mathbb{R}^k$ and by $\norm{\cdot}_2$ the induced norm.  We say that the Frobenius norm of a square matrix $A \in \mathbb{R}^k$, $\norm{A}_F$, is consistent with the Euclidean norm $\norm{x}_2$ if $\norm{Ax}_2 \leq \norm{A}_F\norm{x}_2$. Let $f : E\times \Theta \to \mathbb{R}^k$ be a vector valued function. We denote by $D_{\mathbf{j}}^mf(x, \theta)$ a partial derivative of $m$-th order of $f(x,\theta)$ with respect to the vector of parameters $\theta$, ie.
\begin{align*}
D_{\mathbf{j}}^mf(x,\theta)\coloneqq\left[\frac{\partial^mf_1}{\partial \theta_1^{j_1}\cdots \theta_d^{j_d}}(x,\theta), \,\cdots, \, \frac{\partial^mf_k}{\partial \theta_1^{j_1}\cdots \theta_d^{j_d}}(x,\theta)\right]^T
\end{align*}
where $\mathbf{j}=\left[j_1,j_2,\dots,j_d\right]^T$ and $m=j_1+\dots+j_d$. $\nabla_x f(x, \theta)$ denotes a $k$-dimensional matrix of partial derivatives of $f$ with respect to the space vector $x$. If function $f$ is such that $f: E \to \mathbb{R}^k$, $\nabla f$ is the Jacobian matrix of $f$, and for every component function $f_i$, $\nabla \left(\nabla f_i \right)$ denotes a matrix of second partial derivative of $f_i$. If function $f$ is such that $f: \Theta \to \mathbb{R}$, then $Df(\theta)$ denotes a $d$-dimensional vector of partial derivatives with respect to $\theta$ and $D^2f(\theta)$ a matrix of second partial derivatives with respect to $\theta$. Briefly, $\partial_j f(\theta)$ denotes partial derivative $\partial f(\theta)/\partial \theta_j$ for $j=1,\dots, d$. Vector $e_p$ is a unit vector that has one at $p$-th place and zeros elsewhere.
\begin{lemma}\label{mito}
For $C^2$-function $F : \mathbb{R}^k \rightarrow \mathbb{R}^k$ and random process $(X_t)_{0 \leq t\leq T}$ that satisfies \eqref{sde}, It\^o formula on $[s_1, s_2]$ yields \cite{klopla}:
\begin{align*}
&F(X_{s_2})-F(X_{s_1})=\\
&=\int_{s_1}^{s_2} \left(\nabla F(X_u)\mu(X_u, \theta_0)+ \frac{1}{2}
\nabla_2 F(X_u)\right)\,du+\\
&+ \int_{s_1}^{s_2} \nabla F(X_u) \nu(X_u) 
\, dW_u,
\end{align*}
where $\nabla_2 F$ denotes vector  $\left[\Tr{\left[S(\cdot)\nabla\left(\nabla F_1(\cdot)\right)\right]},\,\dots,\Tr{\left[S(\cdot)\nabla\left(\nabla F_k(\cdot)\right)\right]}\right]^T$.
\end{lemma}
For a scalar function $F$ the previous lemma has the following form.
\begin{lemma}\label{vito}
For $C^2$-function $F : \mathbb{R}^k \rightarrow \mathbb{R}$ and random process $(X_t)_{0\leq t \leq T}$ that satisfies \eqref{sde} It\^o formula yields:
\begin{align*}
&F(X_{s_2})-F(X_{s_1})=\\
&=\int_{s_1}^{s_2} \left(\left<\nabla F(X_u)\,| \,\mu(X_u,\theta_0)\right>+\frac{1}{2}\Tr{\left[S(X_u)\nabla(\nabla F(X_u))\right]}\right)\,du \\
&+ \int_{s_1}^{s_2} \left<\nabla F(X_u)\,|\,\nu(X_u)
\,dW_u\right>.
\end{align*}
\end{lemma}
In the proofs of this paper we often use It\^o isometry 
\cite[Corollary 3.1.7]{oks}, but for vector and matrix random processes. For this reason we formulate it here.
\begin{lemma}\label{iso}
If $M : [0,T]\times \Omega \rightarrow \mathbb{R}^{k\times k}$ is a matrix random process and $V : [0,T]\times \Omega \rightarrow \mathbb{R}^k$ is a vector random process, both adapted to the natural filtration of Brownian motion $W$, then the following holds.
\renewcommand{\theenumi}{\roman{enumi}}
\renewcommand {\labelenumi} {(\roman{enumi})}
\begin{enumerate}
\item \label{miito}
$ \mathbb{E}\left[\norm{\int_0^T M_t\,dW_t}_2^2\right]=\mathbb{E}\left[\int_0^T \norm{M_t}_F^2\,dt\right],$
\item \label{viito}
$\mathbb{E}\left[\left(\int_0^T V_t^T\,dW_t\right)^2\right]=\mathbb{E}\left[\int_0^T \norm{V_t}_2^2\,dt\right].$
\end{enumerate}
\end{lemma}
We also use the generalized version of the mean value theorem given in the following lemma.
\begin{lemma}\label{mvt}
Let $n, m \in \mathbb{N}$. Let $U \subseteq \mathbb{R}^n$ be an open set and $f : U \to \mathbb{R}^m$ be a continuously differentiable function. Let $x \in U$ and $h \in \mathbb{R}^n$ be vectors such that for all $t \in [0,\, 1]$ the line segment $x+th$ lies inside the set $U$. It follows that
\begin{align*}
f(x+h)-f(x)= \left(\int_0^1 \nabla f(x+th)\,dt\right) \cdot h,
\end{align*}
where the integral of a matrix is taken componentwise.
\end{lemma}
For fixed $T>0$, let $\left(\Omega, \mathcal{F}_T, (\mathcal{F})_{0 \leq t \leq T}, \mathbb{P}\right)$ be a given filtered probability space and $\left(\tilde{\Omega}, \tilde{\mathcal{F}}_T, (\tilde{\mathcal{F}})_{0 \leq t \leq T}, \tilde{\mathbb{P}}\right)$ be an \textit{extension} of this space. The extension is called \textit{very good} if all martingales on the initial space are also martingales on the extension \cite{jacod1997}. Let $D$ be a Polish space. In this paper $D$ will be the Skorokhod space $D([0,T], \mathbb{R}^d)$ or $d$-dimensional Euclidean space $\mathbb{R}^d$. Let $(Z_n)_{n \in \mathbb{N}}$ be a sequence of random vectors with values in $D$, defined on the initial space $\left(\Omega, \mathcal{F}_T, (\mathcal{F})_{0 \leq t \leq T}, \mathbb{P}\right)$, and let $Z$ be a random vector also with values in $D$ defined on the extension. We say that $(Z_n)_n$ \textit{converges stably in law} to $Z$, $Z_n \overset{st}{\Rightarrow} Z$, if
\begin{align*}
\lim_{n \to +\infty} \mathbb{E}\left[Yf(Z_n)\right]=\tilde{\mathbb{E}}\left[Yf(Z)\right]
\end{align*}
for all bounded continuous functions $f:D \to \mathbb{R}$ and all bounded random variable $Y$ on $\left(\Omega, \mathcal{F}_T, (\mathcal{F})_{0 \leq t \leq T}, \mathbb{P}\right)$ \cite{jacprot}.\par
We say that an $\mathbb{R}^d$-valued random vector $Y$ has \textit{mixed normal distribution} with $\mathcal{F}_T$-measurable random covariance matrix $C=(C_{jl})$ and write $Y \sim MN(0,C)$ if
\begin{align*}
\mathbb{E}\left[e^{i\left<t\,|\,Y\right>}|\mathcal{F}_T\right]=e^{-\frac{1}{2}\sum_{j,l=1,\dots,d}t_j t_l C_{jl}}.
\end{align*}
Moreover, if $Y \sim MN(0,C)$, then $Y$ has the same distribution as $\sqrt{C}Z$ where $\sqrt{C}$ is a square symmetric root of $C$ and $Z$ is a standard normal random vector independent of $\mathcal{F}_T$.\par
Let us denote by $A_n^t=\max\{j; 0 \leq j \leq n \text{ and } t_j\leq t\}$ and $\mathcal{F}_{n,i}\coloneqq \mathcal{F}_{t_i}, i=0,1,\dots, n$.\par
For a $\mathbb{R}^d$-valued process $(Y_t)_{1\leq t \leq T}$ we say that it is a \textit{centered Gaussian process} if for every $0\leq t_1 < t_2 < \cdots < t_l \leq T$ the random vector $(Y_{t_1}, \dots, Y_{t_l}) \in \mathbb{R}^{dl}$ has $dl$-dimensional normal distribution, and $\mathbb{E}\left[Y_t\right]=\mathbf{0}_d, t \in [0,T]$ \cite{oks}.\par
The decisive role in the future analysis is played by Theorem 3-2 in \cite{jacod1997}, which has been adapted in its notation for our paper.
\begin{theorem}\label{jac}
Let $W$ be a $k$-dimensional Brownian motion on $[0, T]$ and $\chi_i^n$ be $\mathcal{F}_{n,i}$-measurable square-integrable random vectors in $\mathbb{R}^d$. Assume that $C=(C_{jl})$ is a continuous adapted process defined on $(\Omega, \mathcal{F}_T, \mathbb{F}=(\mathcal{F}_t)_{0\leq t \leq T}, \mathbb{P})$, and $(C_{jl}(t))$ is a positive semidefinite symmetric $d \times d$ matrix for every $t \in \left[0,T\right]$. Assume also that:
\renewcommand{\theenumi}{\roman{enumi}}
\renewcommand {\labelenumi} {(\roman{enumi})}
\begin{enumerate}
\item \label{ppt1}
$\sup_{0 \leq t \leq T}\norm{ \sum_{i=1}^{A_n^t} \mathbb{E}\left[\chi_i^n | \mathcal{F}_{n,i-1}\right]}_2 \overset{\mathbb{P}}{\to} 0,$
\item \label{ppt2}
$\sum_{i=1}^{A_n^t}\left(\mathbb{E}\left[\chi_i^{n,j} \chi_i^{n,l}|\mathcal{F}_{n,i-1}\right]-\mathbb{E}\left[\chi_i^{n,j}| \mathcal{F}_{n,i-1}\right]\mathbb{E}\left[\chi_i^{n,l}|\mathcal{F}_{n,i-1}\right]\right) \overset{\mathbb{P}}{\to} C_{jl}(t),\notag\\ \forall t \in [0,T], \quad j,l=1, \dots, d,$
\item \label{ppt3}
$\sum_{i=1}^{A_n^t} \mathbb{E}\left[\chi_i^n \left(W_{t_i}-W_{t_{i-1}}\right)^T|\mathcal{F}_{n,i-1}\right]\overset{\mathbb{P}}{\to} \mathbf{0}_{d \times k}, \forall t \in [0,T],$
\item \label{ppt4}
$\sum_{i=1}^{n} \mathbb{E}\left[\norm{\chi_i^n}_2^2 \mathbbm{1}_{\norm{\chi_i^n}_2>\epsilon} | \mathcal{F}_{n,i-1}\right]\overset{\mathbb{P}}{\to} 0, \forall \epsilon >0,$
\item \label{ppt5}
$\sum_{i=1}^{A_n^t}\mathbb{E}\left[\chi_i^n \left(N_{t_i}-N_{t_{i-1}}\right)| \mathcal{F}_{n,i-1}\right]\overset{\mathbb{P}}{\to} \mathbf{0}_d, \forall t \in [0, T],$
\end{enumerate}
where $N$ is a bounded $\mathcal{F}_t$-martingale orthogonal to $W$, i.e. $N_0=0$ and for all $ m=1,2,\dots,k$, $[N,W^{m}]_t \equiv 0$.\par Then there is a very good extension $\left(\tilde{\Omega}, \tilde{\mathcal{F}}, \tilde{\mathbb{F}}=(\tilde{\mathcal{F}_t})_{0\leq t\leq T}, \tilde{\mathbb{P}}\right)$ of the space\\ $\left(\Omega, \mathcal{F}_T, \mathbb{F}=(\mathcal{F}_t)_{0\leq t \leq T}, \mathbb{P}\right)$ and a continuous process $Y$ defined on that extension that is, conditionally on $\mathcal{F}_T$, centered Gaussian process with independent increments satisfying $\tilde{\mathbb{E}}\left[Y_t^j Y_t^l | \mathcal{F}_T\right]=C_{jl}(t)$, $t \in [0,T]$, $j,l=1,\dots,d$, and such that
\begin{align*}
\sum_{i=1}^{A_n^t} \chi_i^n \overset{st}{\Rightarrow} Y \quad \text{ on } D([0, T], \mathbb{R}^d).
\end{align*}
\end{theorem}
Let $(\gamma_n)_{n \in \mathbb{N}}$ be a sequence of positive numbers and let $(Y_n)_n$ be a sequence of random vectors  defined on some probability space. We say that $(Y_n)_n$ is of order $O_{\mathbb{P}}(\gamma_n)$, and write $Y_n=O_{\mathbb{P}}(\gamma_n)$, if the sequence $(Y_n/\gamma_n)_{n \in \mathbb{N}}$ is bounded in probability, i.e. if
\begin{align*}
\lim_{A \to +\infty} \overline{\lim\limits_{n}}\, \mathbb{P}\{\gamma_n^{-1}\norm{Y_n}_2>A\}=0.
\end{align*}

\section{Estimation method}
Using the Euler discretization scheme, for $l=1,2, \dots, k$ we obtain the system of stochastic difference equations \cite{klopla}:
\begin{align*}
Z_{t_i}^{l}-Z_{t_{i-1}}^{l}=\mu_l(Z_{t_{i-1}}, \theta)(t_i-t_{i-1})+\sum_{j=1}^{k} 
\nu_{lj}(Z_{t_{i-1}})\left(W_{t_i}^j-W_{t_{i-1}}^j\right).
\end{align*}
If we denote by $\Delta W^j_i \coloneqq W_{t_i}^j-W_{t_{i-1}}^j \sim N(0, \Delta_n)$, then $\Delta W^{j_1}_{i_1}$ and $\Delta W^{j_2}_{i_2}$ are independent for every pair of indices $j_1$ and $j_2$ such that $j_1 \neq j_2$, and for every $i_1$ and $i_2$ such that $i_1\neq i_2$ in case when $j_1 =j_2$. If matrix $S$ is a regular matrix and Hermitian, then the log-likelihood function of the process $Z$ is
\begin{align*}
\ln{L_n(\theta)}&\coloneqq-\ln \left(\prod_{i=1}^n \sqrt{2\pi \Delta_n}^k\sqrt{\det (S(Z_{t_{i-1}}))}\right)-\\&-\frac{1}{2\Delta_n}\sum_{i=1}^n \bigl<S^{-1}(Z_{t_{i-1}}) \left( Z_{t_i}-Z_{t_{i-1}}-\Delta_n\mu(Z_{t_{i-1}}, \theta)\right)\,| \,\\ &
Z_{t_i}-Z_{t_{i-1}}-\Delta_n\mu(Z_{t_{i-1}}, \theta)\bigr>.
\end{align*}
If we substitute $Z_i$ with $X_i\equiv X_{t_{i}}$, for every $i=0,1,\dots,n$, use the fact that matrix $S$ is Hermitian, and omit the constant that does not depend on $\theta$, we get an approximate log-likelihood function of the form
\begin{align}\label{lln}
\ell_n(\theta)& \coloneqq\sum_{i=1}^n \left<S^{-1}(X_{i-1})\mu(X_{i-1}, \theta)\,|\, X_i-X_{i-1}\right>-\\\notag
&-\frac{1}{2}\Delta_n\sum_{i=1}^n \left<S^{-1}(X_{i-1})\mu(X_{i-1}, \theta)\,|\, \mu(X_{i-1}, \theta)\right>.
\end{align}
If there is a point of global maximum $\overline{\theta}_n$ of function $\ell_n$, we call it an approximate maximum likelihood estimator (AMLE) of parameter $\theta$.

\section{Assumptions and main theorems}
Let $O^{\Theta}\subseteq \mathbb{R}^d$ be an open set such that $Cl(\Theta)\subseteq O^{\Theta}$. We assume the following hypothesis:
\begin{itemize}
\item[{\bf A1}] (uniform ellipticity on compact sets)
For every compact set $K \subseteq E$ there exists a real constant $\lambda_K > 0$ such that for all $x \in K$ and $\xi \in \mathbb{R}^k$
\begin{align*}
\left<S(x)\xi\,|\, \xi\right> \geq \lambda_K \norm{\xi}_2^2.
\end{align*}
Especially, matrix $S(x)$ is positive definite (and so regular) for every $x\in E$.
\item[{\bf A2}]
For all $p,r=1,2, \dots, k$ and all $(x,\theta )\in E\times O^{\Theta}$ there exist partial derivatives $\frac{\partial}{\partial x_p}\mu(x,\theta)$ and $\frac{\partial^2}{\partial x_p \partial x_r}\mu(x,\theta)$. Moreover,
functions $(x,\theta) \mapsto\mu(x,\theta)$, $(x,\theta)\mapsto \frac{\partial}{\partial x_p}\mu(x,\theta)$, $(x,\theta)\mapsto \frac{\partial^2}{\partial x_p \partial x_r}\mu(x,\theta)$ are continuous functions on $E \times Cl(\Theta)$, for all $p,r=1,2, \dots, k$.
\item[{\bf A3}] For all $i,j=1,2,\ldots, d$, all $p,r=1,2, \dots, k$ and all $(x,\theta )\in E\times O^{\Theta}$ there exist partial derivatives $\frac{\partial}{\partial \theta_i} \mu(x,\theta)$, $\frac{\partial^2}{\partial x_p \partial \theta_i}\mu(x,\theta)$, $\frac{\partial^3}{\partial x_r \partial x_p \partial \theta_i}\mu(x, \theta)$,
     and $\frac{\partial^2}{\partial \theta_i\partial \theta_j} \mu(x,\theta)$, $\frac{\partial^3}{\partial x_p \partial \theta_i\partial \theta_j}\mu(x,\theta)$.
      Moreover,
functions $(x, \theta)\mapsto\frac{\partial}{\partial \theta_i} \mu(x,\theta)$, $(x,\theta) \mapsto \frac{\partial^2}{\partial x_p \partial \theta_i}\mu(x,\theta)$, $(x,\theta) \mapsto \frac{\partial^3}{\partial x_r \partial x_p \partial \theta_i}\mu(x, \theta)$, and 
$(x, \theta) \mapsto \frac{\partial^2}{\partial \theta_i\partial \theta_j} \mu(x,\theta)$, $(x,\theta) \mapsto \frac{\partial^3}{\partial x_p \partial \theta_i\partial \theta_j}\mu(x,\theta)$ are continuous on $E\times Cl(\Theta)$, for all $p,r=1,2, \dots,k$ and all $i,j=1,2,\dots,d$.
\item[{\bf A4}]
For all $p,r=1,2,\ldots,k$ and all $x\in E$ there exist partial derivatives $\frac{\partial}{\partial x_p}\nu(x)$, and
$\frac{\partial^2}{\partial x_r \partial x_p}\nu(x)$. Moreover,
functions $x \mapsto\nu(x)$, $x \mapsto \frac{\partial}{\partial x_p}\nu(x)$, and $x \mapsto \frac{\partial^2}{\partial x_r \partial x_p}\nu(x)$ are continuous on $E$, for all $p,r=1,2,\dots,k$.
\end{itemize}
Assumptions (A1-2) and (A4) imply that there exists a strong and by-path-unique solution to the SDE (\ref{sde}) on time-interval $[0,T]$ (see \cite{khasminskii}), and (A2-4) enable applications of the It\^{o} formula on the drift and diffusion coefficient functions. Notice that (A2-3) implies that all other partial derivatives of orders less or equal 3 exists and among them all partial derivatives with respect to the same set of variable directions are mutually equal. For example, it turns out that
\[ \frac{\partial^3}{\partial \theta_i\partial x_p\partial \theta_j}\mu(x,\theta)=
\frac{\partial^3}{\partial \theta_i\partial \theta_j\partial x_p}\mu(x,\theta)=
\frac{\partial^3}{\partial x_p\partial\theta_i\partial \theta_j}\mu(x,\theta)\]
for all $(x,\theta )\in E\times Cl(\Theta)$ and all $p=1,\ldots,k$ and all $i,j=1,2,\ldots,d$, and all these functions are continuous on $E\times Cl(\Theta)$.

Since matrix $S(x)$ is regular by (A1), 
the log-likelihood function of \eqref{sde} is given by \cite[Theorem 6.4.3]{strvar}
\begin{align*}
\ell(\theta)&\coloneqq \int_0^T \left< S^{-1}(X_s)\mu(X_s,\theta)\,|\, dX_s\,\right>-\\&
-\frac{1}{2}\int_0^T \left<\mu(X_s, \theta)\,|\, S^{-1}(X_s)\mu(X_s, \theta)\right>\,ds.
\end{align*}
By the next assumption we will be able to prove that log-likelihood function $\ell$ is two-times continuously differentiable on $\Theta$, i.e. $\ell\in C^2 (\Theta )$ (see Theorem \ref{tm_nepr} below).
\begin{itemize}
\item[{\bf A5}] For all $(x,\theta) \in E \times O^{\Theta}$ and $m \leq d+3$ there exist $D_{\mathbf{j}}^m \mu(x,\theta)$ and $\frac{\partial}{\partial x_p} D_{\mathbf{j}}^m \mu(x,\theta)$, for all $p=1,\dots,k$ and all $\mathbf{j}=\left[j_1,j_2,\dots,j_d\right]^T$ such that $m=j_1+\dots+j_d$. Moreover, these functions are continuous on $E \times Cl(\Theta)$.
\end{itemize}

\begin{theorem}\label{tm_nepr}
Assume that (A1) and (A4-5) hold. Then the log-likelihood function $\ell\in C^2(\Theta )$, and
\begin{align*}
\frac{\partial}{\partial\theta_i}\ell(\theta)& = \int_0^T \left< S^{-1}(X_s)\frac{\partial}{\partial\theta_i}\mu(X_s,\theta)\,|\, dX_s\,\right>-\\&
-\int_0^T \left<\frac{\partial}{\partial\theta_i}\mu(X_s, \theta)\,|\, S^{-1}(X_s)\mu(X_s, \theta)\right>\,ds\\
\frac{\partial^2}{\partial\theta_i\partial\theta_j}\ell(\theta)& = \int_0^T \left< S^{-1}(X_s)\frac{\partial^2}{\partial\theta_i\partial\theta_j}\mu(X_s,\theta)\,|\, dX_s\,\right>-\\&
-\int_0^T \left<\frac{\partial^2}{\partial\theta_i\partial\theta_j}\mu(X_s, \theta)\,|\, S^{-1}(X_s)\mu(X_s, \theta)\right>\,ds -\\&
-\int_0^T \left<\frac{\partial}{\partial\theta_i}\mu(X_s, \theta)\,|\, S^{-1}(X_s)\frac{\partial}{\partial\theta_j}\mu(X_s, \theta)\right>\,ds
\end{align*}
for all $i,j=1,2,\ldots,d$.
\end{theorem}

Now, we can formulate the last assumption:
\begin{itemize}
\item[{\bf A6}]
For all $\omega \in \Omega$, the log-likelihood function $\ell(\theta)\equiv \ell_T(\theta, \omega)$ based on continuous observations of $X$  over $\left[0,T\right]$, as a function $\ell : \Theta \to \mathbb{R}$, has a unique point of global maximum $\hat{\theta}\equiv \hat{\theta}_T(\omega)$, and $D^2 \ell(\hat{\theta})$ is a negatively definite matrix.
\end{itemize}

Since $\Theta$ is a convex space, using the mean value theorem, we get that for all $\theta \in \Theta$
\begin{align*}
D\ell_n(\theta)=D\ell_n(\hat{\theta})+\int_0^1 D^2\ell_n(\hat{\theta}+(\theta-\hat{\theta})z)\,dz \cdot (\theta-\hat{\theta}).
\end{align*}
Especially, on event $\{D\ell_n(\overline{\theta}_n)=0\}$, we have
\begin{align*}
\frac{1}{\sqrt{\Delta_n}}\int_0^1 &D^2 \ell_n(\hat{\theta}+(\overline{\theta}_n-\hat{\theta})z)\,dz \cdot (\overline{\theta}_n-\hat{\theta})=\notag\\
&=\frac{1}{\sqrt{\Delta_n}}\left(D\ell(\hat{\theta})-D\ell_n(\hat{\theta})\right)
\end{align*}
since $D\ell (\hat{\theta})=0$ by (A6).
We are interested in the behavior of the difference between the derivative of the log-likelihood function for continuous observations and its discretized version.

Let us fix $\theta \in \Theta$. For $1 \leq j \leq d$, we have
\begin{align}
\frac{1}{ \sqrt{\Delta_n}}\left( D\ell(\theta)-D\ell_n(\theta)\right)_j&=\frac{1}{\sqrt{\Delta_n}}\sum_{i=1}^n \int_{t_{i-1}}^{t_i} \bigl< S^{-1}(X_s)\partial_j \mu(X_s,\theta)-\label{difference}\\
&-S^{-1}(X_{i-1})\partial_j \mu(X_{i-1}, \theta)\,|\, \mu(X_s, \theta_0)\bigr>\,ds -\label{a1} \\ &-\frac{1}{\sqrt{\Delta_n}}\sum_{i=1}^n \int_{t_{i-1}}^{t_i} \Bigl(\left<\partial_j \mu(X_s, \theta)\,|\, S^{-1}(X_s)\mu(X_s, \theta)\right>-\notag\\
&-\left<\partial_j \mu(X_{i-1}, \theta)\,|\, S^{-1}(X_{i-1})\mu(X_{i-1}, \theta)\right> \Bigr)ds + \label{a2} \\
& + \frac{1}{\sqrt{\Delta_n}} \sum_{i=1}^n \int_{t_{i-1}}^{t_i} \bigl< S^{-1}(X_s)\partial_j \mu(X_s,\theta)-\notag\\
&-S^{-1}(X_{i-1})\partial_j \mu(X_{i-1}, \theta)\,|\, \nu(X_t)
\,dW_s\bigr>\label{a3} .
\end{align}
We introduce the following notation:
\begin{align}\label{gj}
g_j(x, \theta) \coloneqq S^{-1}(x) \partial_j\mu(x, \theta),\;\; (x,\theta )\in E\times\Theta.
\end{align}
First, the difference in \eqref{difference} can be decomposed into two parts. One part is negligible in probability and the other part determines the resulting limit.
\begin{lemma}\label{lemma1}
Assume that (A1-5) hold and that $E$ is a compact set. For arbitrary $\theta \in \Theta$ the standardized difference
\begin{align*}
\frac{1}{\sqrt{\Delta_n}}\left(D\ell(\theta)-D\ell_n(\theta)\right)
\end{align*}
is equal to
\begin{align*}
V_n(\theta)+\frac{1}{\sqrt{\Delta_n}}
\begin{bmatrix}
\sum_{i=1}^n \int_{t_{i-1}}^{t_i} \left<\int_{t_{i-1}}^s \nabla_x g_1(X_u, \theta)\nu(X_u)
\, dW_u\,| \, \nu(X_s)
\, dW_s\right>\\
\vdots\\
\sum_{i=1}^n \int_{t_{i-1}}^{t_i} \left<\int_{t_{i-1}}^s \nabla_x g_d(X_u, \theta)\nu(X_u)
\, dW_u\,| \, \nu(X_s)
\, dW_s\right>
\end{bmatrix},
\end{align*}
where $V_n(\theta) \overset{\mathbb{P}}{\to} \mathbf{0}_d$ when $n \to +\infty$.
\end{lemma}
\begin{theorem}\label{main1}
Assume that (A1-5) hold. For arbitrary $\theta \in \Theta$ the standardized difference
\begin{align*}
\frac{1}{\sqrt{\Delta_n}}\left(D\ell(\theta)-D\ell_n(\theta)\right) \overset{st}{\Rightarrow} Y(\theta)
\end{align*}
where $Y(\theta) \sim MN(0, \Sigma(\theta))$ and for $j, l=1,\dots, d$, $$\Sigma_{jl}(\theta)=\frac{1}{2}\int_0^T \sum_{p,r=1}^k S_{pr}(X_s) \left<e_r^T \nabla_x g_j(X_s, \theta)\nu(X_s)
\,| \, e_p^T \nabla_x g_l(X_s, \theta)\nu(X_s)
\right>\,ds.$$
\end{theorem}
%
%
%
%
%
%
%
Compactness of $Cl(\Theta )$ implies that is possible to pass MLE $\hat{\theta}$ in the limiting equation in the preceding theorem.
\begin{theorem}\label{main2}
Assume that (A1-6) holds. Assume that $E$ is an open convex set and
$\Theta$ is a convex and relatively compact set. Then, we have
\begin{align*}
\frac{1}{\sqrt{\Delta_n}}\left(D\ell(\hat{\theta})-D\ell_n(\hat{\theta})\right) \overset{st} \Rightarrow Y(\hat{\theta})
\end{align*}
where $Y$ is the same as in Theorem \ref{main1}.
\end{theorem}
To analyze the difference between the AMLE and MLE, we need to obtain the uniform bound on the difference between the derivative of the log-likelihood and the discretized log-likelihood. The following lemma serves this purpose.
\begin{lemma}\label{konstr}
Assume that (A1-5) hold. Let $\Theta$ be a relatively compact set. Then for $r=0,1,2$
\begin{align*}
\sup_{\theta \in \Theta}\norm{D^r\ell_n(\theta)-D^r\ell(\theta)}_2=O_{\mathbb{P}}(\sqrt{\Delta_n}),\quad n \in \mathbb{N}.
\end{align*}
\end{lemma}
Using Lemma \ref{konstr} and assumption (A6), it can be proved (exactly as in \cite{huzak2001}) that there exists a sequence of AMLEs. Namely, the following theorem holds.
\begin{theorem}\label{ex_amle}
Assume that (A1-6) hold. Then there exists a sequence
$(\overline{\theta}_n)_{n \in \mathbb{N}} \subseteq \Theta$ of $\mathcal{F}_T$-measurable random vectors such that
\begin{itemize}
\item[(i)]
$\lim_{n \to +\infty} \mathbb{P}\left(D\ell_n(\overline{\theta}_n)=0\right)=1$,
\item[(ii)]
$\overline{\theta}_n \overset{\mathbb{P}}{\to} \hat{\theta}$, $n \to +\infty$,
\item[(iii)]
if $(\tilde{\theta}_n)_{n \in \mathbb{N}}$ is another sequence of random vectors which satisfies (i) and (ii), then $\mathbb{P}(\tilde{\theta}_n=\overline{\theta}_n)=1$,
\item[(iv)]
$\left(\dfrac{1}{\sqrt{\Delta_n}}\norm{\overline{\theta}_n-\hat{\theta}}_2\right)_n$ is bounded in probability.
\end{itemize}
\end{theorem}
The core of our paper is the following theorem, which gives the limit of the standardized difference between the AMLE and MLE in terms of stable convergence. From now on we assume that the vector $\overline{\theta}_n$ is such that (i)-(iv) from Theorem \ref{ex_amle} hold.
\begin{theorem}\label{the_core}
Let $\Theta$ be a convex and relatively compact set. Assume (A1-6). Then
\begin{align*}
\frac{1}{\sqrt{\Delta_n}}\left(\overline{\theta}_n-\hat{\theta}\right) \overset{st}{\Rightarrow} MN\left(0, \left(D^2\ell(\hat{\theta})\right)^{-1}\Sigma(\hat{\theta})\left(D^2\ell(\hat{\theta})\right)^{-1}\right).
\end{align*}
\end{theorem}

\section{Proofs}
For proving Theorem \ref{tm_nepr} and Lemma \ref{konstr} we need help of Fourier analysis.

Let $U\subset O^{\Theta}$ be an open and bounded set such that the following holds:
\[ U\subset Cl(U)\subset U'\subset Cl(U')\subset O^{\Theta}\]
for some $d$-dimensional rectangular $U'=\prod_{i=1}^d\langle a_i,b_i\rangle$.
Since $O^{\Theta}\subset\mathbb{R}^d$ is an open and locally compact set, for every $\theta\in O^{\Theta}$ there exists such an open neighborhood $U$ of $\theta$ in $O^{\Theta}$. On the other hand, if $\mathcal{K}$ is a relatively compact open set in $\Theta$, then it can be covered by a finite number of such open sets $U_j$ ($j$ from a finite set).
If $f:E\times O^{\Theta}\rightarrow\mathbb{R}^k$ is a bounded function such that $f(x,\cdot )\in C^m (Cl (\Theta ))$ 
for some $m$ and all $x\in E$,
then there exists a bounded function $\tilde{f}:E\times\mathbb{R}^d\rightarrow\mathbb{R}^k$ and a regular linear mapping
$A:\mathbb{R}^d \rightarrow\mathbb{R}^d$ such that $\tilde{f}(x,\cdot )\in C^m (\mathbb{R}^d)$ for all $x\in E$ and
$A( Cl(U'))=[-\pi, \pi]^d$, and
\begin{align*}
&(\forall(x,\theta )\in E\times Cl(\mathcal{K}))\;\; f(x,\theta )=\tilde{f}(x,A\theta )\\
&(\forall \theta\in \partial ([-\pi, \pi]^d))\, (\forall x\in E)\;\; \tilde{f}(x,\theta )=0
\end{align*}
where $\partial\mathcal{S}$ denotes boundary of a set $\mathcal{S}$.
Such a function $\tilde{f}$ can be constructed from $f$ by use of an appropriate test-function (see \cite{huzak2018}). For all of above reasons it is enough to prove the following two technical theorems for neighborhoods and relative compacts of the form $\mathcal{K}_0:=\left<-\pi, \pi\right>^d$ and functions that satisfy assumptions
(P1) and (P2) below.

Let $\mathbf{k}=[k_1, \dots, k_d]^T \in \mathbb{Z}^d$, $\mathbf{j}=[j_1,\dots, j_d]^T \in \mathbb{N}_0^d$ and $m=|\mathbf{j}|:=j_1+\cdots+j_d \leq d+1$. For function $f:E\times\mathbb{R}^d
\rightarrow\mathbb{R}^k$ we assume the following: 
\begin{itemize}
\item[{\bf P1}]
For all $\mathbf{j}\in \mathbb{N}_0^d$ such that $|\mathbf{j}|=m\leq d+1$, and all $(x,\theta )\in E\times\mathbb{R}^d
$ there exist $\frac{\partial f}{\partial x_p}(x,\theta), D_{\mathbf{j}}^m f(x,\theta)$ and $\frac{\partial}{\partial x_p}D_{\mathbf{j}}^m f(x,\theta)$ for $p=1,\dots, k$. Additionally, $f$, $\frac{\partial f}{\partial x_p}(x,\theta)$, $D_{\mathbf{j}}^m f(x,\theta)$ and $\frac{\partial}{\partial x_p}D_{\mathbf{j}}^m f(x,\theta)$, $m\leq d+1$, $p=1,\dots, k$, are continuous and uniformly bounded functions on 
$E \times\mathbb{R}^d 
$.
\item[\bf{P2}]
For all $x \in E$ and all $\mathbf{j}\in \mathbb{N}_0^d$ such that $|\mathbf{j}|=m\leq d+1$, $f(x,\cdot)\equiv \mathbf{0}_{k \times k}$
and $D_{\mathbf{j}}^m f(x,\cdot)\equiv \mathbf{0}_{k}$
on $\partial \mathcal{K}_0$.
\end{itemize}
Fourier coefficients of function $f$ are defined (componentwise) as vectors
\begin{align*}
C_{\mathbf{k}}(x)&\coloneqq \frac{1}{(2\pi)^d}\int_{Cl(\mathcal{K}_0)} f(x,\theta)e^{-i\left<\mathbf{k}|\theta\right>}\,d\theta,\\
C_{\mathbf{k}}^{\mathbf{(j)}}(x)&\coloneqq \frac{1}{(2\pi)^d}\int_{Cl(\mathcal{K}_0)} D_{\mathbf{j}}^mf(x,\theta)e^{-i\left<\mathbf{k}|\theta\right>}\,d\theta.
\end{align*}
Let us denote $\mathbf{k}^{\mathbf{j}}\coloneqq k_1^{j_1}\cdots k_d^{j_d}$. Under assumptions (P1) and  (P2) using partial integration it is easy to see that for all $x \in E$, we have \cite{taylor}
\begin{align}\label{sv_fk}
C_{\mathbf{k}}^{(\mathbf{j})}(x)=i^m \mathbf{k}^{\mathbf{j}}C_{\mathbf{k}}(x).
\end{align}
\begin{remark}\label{four}
If assumption (P2) does not hold for function $f$, we can easily calculate the relation similar to \eqref{sv_fk}, and it is of the form
\begin{align*}
F_{\mathbf{k}}(x)+C_{\mathbf{k}}^{(\mathbf{j})}(x)=i^m \mathbf{k}^{\mathbf{j}}C_{\mathbf{k}}(x),
\end{align*}
where $F_{\mathbf{k}}(x)$ is a linear combination of functions $D_{\theta}^{m-1}f(x,\theta)$, for $1\leq m\leq d+1$, on $\partial \mathcal{K}_0$. Since these functions are bounded on $E\times\mathbb{R}^d
$ by (P1), expression $F_{\mathbf{k}}(x)$ can be bounded by a constant so it may be ignored in further considerations (see the proof of Lemma \ref{fk}).
\end{remark}
The following is a list of technical statements essential for proving the uniform (in $\theta$) $L^2$-bounds of sum of certain Riemann and It\^ o integrals as well as absolute convergence of Fourier series
\begin{align}\label{fs0}
\sum_{\mathbf{k}\in\mathbb{Z}^d} C_{\mathbf{k}}(x)e^{i\left<\mathbf{k}|\theta\right>}
\end{align}
for $(x,\theta )\in E\times Cl(\mathcal{K}_0)$.
\begin{lemma}\label{pl}
Let functions $\mu (\cdot,\theta_0)$ and $\nu$ are bounded and continuous on $E$. There exists a constant $K>0$ such that for $s_1$, $s_2$, $0\leq s_1 \leq s_2 \leq T$
$$\mathbb{E}\left[\norm{X_{s_2}-X_{s_1}}_2^2\right]\leq 2K\left(\left(s_2-s_1\right)^2+\left(s_2-s_1\right)\right).$$
\end{lemma}	
\begin{lemma}\label{fk}
Let $E\subseteq \mathbb{R}^k$ be an open and convex set and $x, y \in E$,
and let $f:E\times\mathbb{R}^d
\rightarrow\mathbb{R}^k$ be a function that satisfies (P1) and (P2).
Then for every $\mathbf{k} \in \mathbb{Z}^d$ there exist constants $k_1, k_2>0$ such that
\begin{align}\label{fk1}
\norm{C_{\mathbf{k}}(x)-C_{\mathbf{k}}(y)}_2\leq k_1 K_m (\mathbf{k})\norm{x-y}_2,
\end{align}
\begin{align}\label{fk2}
\norm{C_{\mathbf{k}}(x)}_2\leq k_2 K_m (\mathbf{k}),
\end{align}
where $K_m(\mathbf{k})=\left(\frac{d+1}{1+|k_1|+\cdots+|k_d|}\right)^m$.
\end{lemma}
\begin{lemma}\label{fk3}
$\sum_{\mathbf{k}\in \mathbb{Z}^d} \frac{1}{(1+|k_1|+\cdots+|k_d|)^{d+1}} < +\infty.$
\end{lemma}
If $f$ satisfies (P1) and (P2), then 
\begin{align*} 
&\sum_{\mathbf{k}\in\mathbb{Z}^d}\norm{C_{\mathbf{k}}(x)e^{i\left<\mathbf{k}|\theta\right>}}_2\leq 
 \sum_{\mathbf{k}\in\mathbb{Z}^d}\norm{C_{\mathbf{k}}(x)}_2\leq k_2 \sum_{\mathbf{k}\in\mathbb{Z}^d}K_{d+1}(\mathbf{k})=\\
 &= k_2 (d+1)^{d+1}\sum_{\mathbf{k}\in \mathbb{Z}^d} \frac{1}{(1+|k_1|+\cdots+|k_d|)^{d+1}} < +\infty,
 \end{align*}
 by Lemma \ref{fk} and Lemma \ref{fk3},  i.e.\ Fourier series \eqref{fs0} absolutely converges and hence it converges to $f(x,\theta )$ since $f(x,\cdot )$ is continuous by (P1) (see \cite{taylor}, pp.\ 197-206).
 
\begin{theorem}\label{nepr}
Let $\Theta\subseteq\mathbb{R}^d$ be an open set and let $f:E\times \Theta\rightarrow\mathbb{R}$ be a uniformly bounded function such that for all $x\in E\subseteq\mathbb{R}^k$, $f(x,\cdot )\in C^{d+2}(\Theta )$, and all partial derivatives of $f$ of the orders $m\leq d+2$ are uniformly bounded. Moreover, let $\left(X_t\right)_{1\leq t \leq T}$ be a random process that satisfies (\ref{sde}) and let $B$ be a standard Brownian motion defined on the same filtered probability space. Then there exists random function $F:\Omega\times\Theta\rightarrow\mathbb{R}$ such that for all $\omega\in\Omega$, $F(\omega,\cdot )\in C^1 (\Theta )$, and for all $\theta\in\Theta$ (with $F(\theta )\equiv F(\cdot,\theta )$)
\begin{align*}
F(\theta )&=\int_0^T f(X_t,\theta )\, dB_t\;\mbox{\rm a.s.}\\
\frac{\partial}{\partial\theta_j}F(\theta )&=\int_0^T \frac{\partial}{\partial\theta_j}f(X_t,\theta )\, dB_t
\;\mbox{\rm a.s.\ for}\; j=1,2,\ldots ,d.
\end{align*}
\end{theorem}
\begin{proof} [Proof of Theorem \ref{nepr}]
Let $\theta'\in\Theta$ be an arbitrary point. Since $\Theta$ is an open set, there exists open neighborhood $U'\subset \Theta$ of $\theta'$ which is $d$-dimensional rectangular such that $Cl(U')\subset \Theta$. Without loss of generality, let us assume that $U'=\mathcal{K}_0$, and that $f$ can be extended on $E\times\mathbb{R}^d$ such that it satisfies (P1) and (P2).\\
Let $C_\mathbf{k}(x)=(1/(2\pi)^d)\int_{Cl (\mathcal{K}_0)} f(x,\theta )e^{-i\left<\mathbf{k}|\theta\right>}d\theta$ be a Fourier coefficient of a scalar function $f$ for $\mathbf{k}\in\mathbb{Z}^d$ such that $|\mathbf{k}|\leq m=d+1< d+2$. Then $C_\mathbf{k}(\cdot )$ is a bounded and continuous function on $E$ by Lemma \ref{fk} that implies that stochastic 
integral $\int_0^T C_\mathbf{k} (X_t)\, dB_t$ is well defined.
It follows that
\begin{align*}
\mathbb{E}\left( \sum_{\mathbf{k}\in\mathbb{Z}^d}\left|\int_0^T\!\! C_\mathbf{k} (X_t)\, dB_t\right|\right) &\leq 
\sum_{\mathbf{k}\in\mathbb{Z}^d}\norm{\int_0^T\!\! C_\mathbf{k} (X_t)\, dB_t}_{L^2}=\\
&= \sum_{\mathbf{k}\in\mathbb{Z}^d}\left(\mathbb{E}\int_0^T\!\! C_\mathbf{k}^2 (X_t)\, dt\right)^{\frac{1}{2}} \leq \\
&\leq\sqrt{T} k_2 \sum_{\mathbf{k}\in \mathbb{Z}^d} K_{d+1}(\mathbf{k}) < +\infty
\end{align*}
by Lemmas \ref{fk} and \ref{fk3} implying that
\[ \sum_{\mathbf{k}\in\mathbb{Z}^d}\left|\int_0^T C_\mathbf{k} (X_t)\, dB_t\cdot e^{i\left<\mathbf{k}|\theta\right>}\right|\leq \sum_{\mathbf{k}\in\mathbb{Z}^d}\left|\int_0^T C_\mathbf{k} (X_t)\, dB_t\right|<+\infty\;\mbox{\rm a.s.}
\]
Hence series $\sum_{\mathbf{k}\in\mathbb{Z}^d}\int_0^T C_\mathbf{k} (X_t)\, dB_t\cdot e^{i\left<\mathbf{k}|\theta\right>}$ converges absolutely and uniformly for all $\theta\in \mathcal{K}_0$ on event
$\Omega_0:=\{\sum_{\mathbf{k}\in\mathbb{Z}^d}\left|\int_0^T C_\mathbf{k} (X_t)\, dB_t\right|<+\infty\}$ of 
probability 1. Then for all fixed $\omega\in\Omega_0$, its sum 
\[ F(\omega,\theta ):=\sum_{\mathbf{k}\in\mathbb{Z}^d}\int_0^T C_\mathbf{k} (X_t)\, dB_t (\omega )\cdot e^{i\left<\mathbf{k}|\theta\right>},\; \theta\in\mathcal{K}_0,\]
is a continuous function on $\mathcal{K}_0$. Moreover, 
\begin{align*}
\sum_{\mathbf{k}\in\mathbb{Z}^d}\left|\frac{\partial}{\partial\theta_j}\left(\int_0^T C_\mathbf{k} (X_t)\, dB_t \cdot e^{i\left<\mathbf{k}|\theta\right>}\right)\right|&= \sum_{\mathbf{k}\in\mathbb{Z}^d}\left|\int_0^T iC_\mathbf{k} (X_t)k_j\, dB_t \right|\stackrel{(\ref{sv_fk})}{=} \\
&=\sum_{\mathbf{k}\in\mathbb{Z}^d}\left|\int_0^T C_\mathbf{k}^{(e_j)} (X_t)\, dB_t \right|
\end{align*}
for all $j=1,2,\ldots,d$, and hence functional series $\sum_{\mathbf{k}\in\mathbb{Z}^d}\frac{\partial}{\partial\theta_j}(\int_0^T C_\mathbf{k} (X_t)\, dB_t \cdot e^{i\left<\mathbf{k}|\theta\right>})$ converges absolutely and uniformly on event
\[\Omega_j :=\left\{ \sum_{\mathbf{k}\in\mathbb{Z}^d}\left|\int_0^T C_\mathbf{k}^{(e_j)} (X_t)\, dB_t \right|
<+ \infty\right\}.\]
Since $C_\mathbf{k}^{(e_j)}(x)=(1/(2\pi)^d)\int_{Cl (\mathcal{K}_0)} \frac{\partial}{\partial\theta_j}f(x,\theta )e^{-i\left<\mathbf{k}|\theta\right>}d\theta$ and $\frac{\partial}{\partial\theta_j}f$ is a bounded function, 
it follows that $\mathbb{P}(\Omega_j)=1$, and hence  
$F(\omega,\cdot)\in C^1(\mathcal{K}_0)$ and 
\[\frac{\partial}{\partial\theta_j}F(\omega,\theta )=\sum_{\mathbf{k}\in\mathbb{Z}^d}\int_0^T C_\mathbf{k}^{(e_j)} (X_t)\, dB_t (\omega )\cdot e^{i\left<\mathbf{k}|\theta\right>}\]
for all $\omega\in\Omega_0\cap\bigcap_{j=1}^d\Omega_j$ that is an event of probability 1.\\
Let $S_N^{(0)}(x,\theta ):=\sum_{|\mathbf{k}|\leq N}C_\mathbf{k}(x)e^{i\left<\mathbf{k}|\theta\right>}$ and 
$S_N^{(j)}(x,\theta ):=\sum_{|\mathbf{k}|\leq N}C_\mathbf{k}^{(e_j)}(x)e^{i\left<\mathbf{k}|\theta\right>}$ $(N\in\mathbb{N}$) be partial sums of corresponding Fourier series ($j=1,2,\ldots,d$). Since 
$S_N^{(0)}(X_t,\theta )- f(X_t,\theta)$ and $S_N^{(j)}(X_t,\theta )-\frac{\partial}{\partial\theta_j}f(X_t,\theta)$
are uniformly bounded with constants by Lemmas \ref{fk} and \ref{fk3}, and
\[\lim_{N \to +\infty} S_N^{(0)}(X_t,\theta )=f(X_t,\theta),\;\mbox{\rm and}\; \lim_{N \to +\infty} S_N^{(j)}(X_t,\theta )=\frac{\partial}{\partial\theta_j}f(X_t,\theta),\]
for all $(t,\theta )\in [0,T]\times\mathcal{K}_0$, it follows that
\begin{align*}
\int_0^T S_N^{(0)}(X_t,\theta )\,dB_t &\overset{\mathbb{P}}{\to}\int_0^T f(X_t,\theta)\, dB_t,\, N \to +\infty,\\
\int_0^T S_N^{(j)}(X_t,\theta )\,dB_t & \overset{\mathbb{P}}{\to} \int_0^T \frac{\partial}{\partial\theta_j}f(X_t,\theta)\,dB_t, \, N \to +\infty,
\end{align*}
by the dominated convergence theorem for stochastic integrals \cite[Theorem 2.12]{ry}. On the other hand, since 
$\int_0^T S_N^{(0)} (X_t,\theta )\, dB_t$ and $\int_0^T S_N^{(j)} (X_t,\theta )\, dB_t$ are partial sums of 
functional series that a.s.\ converge to $F(\cdot,\theta )$ and $\frac{\partial}{\partial\theta_j}F(\cdot,\theta )$ respectively, it follows that
\[F(\cdot,\theta )=\int_0^T f(X_t,\theta)\, dB_t\;\mbox{\rm a.s.,}\;\; \mbox{\rm and}\;\; 
\frac{\partial}{\partial\theta_j}F(\cdot,\theta )=\int_0^T \frac{\partial}{\partial\theta_j}f(X_t,\theta)\,dB_t
\;\mbox{\rm a.s.}\]
We can extend definition of $F$ on $\Theta $ in the following way. Let $\{ U_j : j\in\mathbb{N}\}$ be a countable open covering of $\Theta $ such that $U_j\subset Cl(U_j)\subset \Theta$ be an open rectangular. If  $F_{U_j}$ is the function defined by Fourier series like above, then $F(\omega,\theta ):=F_{U_j}(\omega, \theta)$ if $\theta\in U_j$ and  for $\omega$ from the probability-one-event obtained as countable intersections of probability-one-events like above. If there exists $j'\neq j$ such that $\theta\in U_j\cap U_{j'}$, then 
$F_{U_{j'}}(\theta )=\int_0^T f(X_t,\theta )\, dB_t =
F_{U_j}(\theta )$ a.s. This implies that $F$ is correctly defined.
\end{proof}

\begin{proof} [Proof of Theorem \ref{tm_nepr}] 
The conclusions follow directly by applying \eqref{sde} for $\theta=\theta_0$ and then Theorem \ref{nepr} on It\^{o} stochastic integrals with respect to components of k-dimensional Brownian motion $W$, and from the fact that integrals with respect to the time variable and partial derivatives with respect to $\theta$ commute.
\end{proof}

\begin{theorem}\label{unif_b}
Let $\left(X_t\right)_{1\leq t \leq T}$ be a random process that satisfies \eqref{sde}. Let $E \subseteq \mathbb{R}^k$ and $\Theta \subseteq \mathbb{R}^d$ be open and convex sets. Moreover, let $f: E\times \Theta \to \mathbb{R}^k$ be a function that satisfies (P1). We assume that $\mu(\cdot, \theta_0)$ and $\nu$ are bounded and continuous on $E$.\\
Let $a : E \to \mathbb{R}^k$ be a bounded vector valued function. For every relatively compact set $\mathcal{K} \subset \Theta$ there exist constants $C_1, C_2>0$ such that:
\begin{align}\label{o_dt}
\norm{\sup_{\theta \in \Theta} \Big|\sum_{i=1}^n\int_{t_{i-1}}^{t_i} \left<f(X_t,\theta)-f(X_{i-1},\theta)\,\big| \, a(X_t)\right>\,dt\Big|}_{L^2}\leq C_1 \sqrt{\Delta_n}
\end{align}
\begin{align}\label{o_dw}
\norm{\sup_{\theta \in \mathcal{K}} \Big|\sum_{i=1}^n\int_{t_{i-1}}^{t_i} \left<f(X_t,\theta)-f(X_{i-1},\theta)\,\big| \, \nu(X_t)
\,dW_t\right>\Big|}_{L^2}\leq C_2 \sqrt{\Delta_n}.
\end{align}
\end{theorem}


\begin{proof} [Proof of Theorem \ref{unif_b}]
Let $c_{20}, \dots, c_{24}$ denote 
 positive constants occurring in the proof. Using Cauchy-Schwarz inequality for integrals and vectors, Lemma \ref{mvt} and boundedness of functions $\nabla_xf$ and $a$, we have that there exists a constant $c_{20}$ such that
\begin{align*}
&\mathbb{E}\left[\left(\sup_{\theta \in \Theta}\Big|\sum_{i=1}^n \int_{t_{i-1}}^{t_i} \left<f(X_t, \theta)-f(X_{i-1}, \theta)\,\big| \, a(X_t)\right>\,dt \Big|\right)^2\right]\leq\\
&\leq c_{20} \mathbb{E}\left[\sum_{i=1}^n  \int_{t_{i-1}}^{t_i} \norm{X_t-X_{i-1}}_2^2\,dt\right].
\end{align*}
For final step, we use Lemma \ref{pl}.
\begin{align}\label{end}
 c_{20}  \sum_{i=1}^n \int_{t_{i-1}}^{t_i} \mathbb{E}\left[\norm{X_t-X_{i-1}}_2^2\right]\,dt&\leq  c_{20}  \sum_{i=1}^n  \int_{t_{i-1}}^{t_i} 2K \bigl((t-t_{i-1})^2+\notag\\
 &+(t-t_{i-1})\bigr)\,dt\leq\notag\\
&\leq c_{20} K \left(\frac{2}{3}T^2+T\right) \Delta_n=\notag\\
&\eqqcolon C_1^2 \Delta_n
\end{align}
so inequality \eqref{o_dt} is proved. 

Next, we prove inequality \eqref{o_dw} by following the arguments from part of the proof of Theorem 6.1 in \cite{huzak2018}. The first main difference is that Brownian motion $W$ is multidimensional, so the first part of the proof is derived componentwise, and then the continuous mapping theorem is used. Also, the observed time interval is fixed compared to the proof of Theorem 6.1 in \cite{huzak2018}, where this is not the case. For this reason, it is easier to justify the application of the dominated convergence theorem for stochastic integrals to the integrals defined below. Moreover, for the same reason, it is easier to obtain the $L^2$-bound of the right-hand side of the inequality in \eqref{red}. To clarify the arguments presented and to distinguish details due to multidimensionality, we provide the entire proof.

Without loss of generality, it is sufficient to prove it for $\mathcal{K}=\mathcal{K}_0$.\\
Let $S_N(x, \theta)=\sum_{|\mathbf{k}|\leq N} C_{\mathbf{k}}(x)e^{i\left<\mathbf{k}|\theta\right>}$ be $N$-th partial sum of Fourier's series of functions $f(x, \theta)$. For fixed $\theta \in \mathcal{K}_0$ and $N \in \mathbb{N}$, let us define the following random processes
\begin{align*}
V_t &\coloneqq \sum_{l=1}^k \sum_{i=1}^n \left<f(X_t,\theta)-f(X_{i-1}, \theta)\,\big| \, \nu(X_t)e_l\right> \mathbbm{1}_{\left<t_{i-1},\, t_i\right]}(t),\quad t \in [0, \,T],\\
V_t^{(N)} &\coloneqq \sum_{l=1}^k \sum_{i=1}^n \left<S_N(X_t, \theta)-S_N(X_{i-1}, \theta)\,\big| \, \nu(X_t)e_l\right> \mathbbm{1}_{\left<t_{i-1},\, t_i\right]}(t), \quad t \in [0,T].
\end{align*}
Moreover, let us define $M_N \coloneqq \sup_{x \in E, \theta \in Cl(\mathcal{K}_0)} \norm{S_N(x,\theta)-f(x, \theta)}_2$ and \\$M \coloneqq \sup_{x \in E} \sum_{\mathbf{k} \in \mathbb{Z}^d} \norm{C_{\mathbf{k}} (x)}_2$.
For $m=d+1$, using \eqref{fk2} in Lemma \ref{fk} and Lemma \ref{fk3} we conclude that $M$ is finite. \\For every $N \in \mathbb{N}$
\begin{align}\label{tdk1}
\Big|V_t^{(N)} \Big| &\leq \sum_{l=1}^k \sum_{i=1}^n \norm{\sum_{|\mathbf{k}|\leq N}\left(C_{\mathbf{k}}(X_t)-C_{\mathbf{k}}(X_{i-1})\right)}_2 \norm{\nu(X_t)e_l}_2 \mathbbm{1}_{\left<t_{i-1}, t_i\right]}(t)\leq\notag\\
&\leq  2M c_{21}
\end{align}
where $c_{21}=\max_{l=1,\dots, k}\norm{\nu(X_t)e_l}_2$.\\
Because of smoothness of function $f$, it can be shown that for fixed $x \in E$,\\ $\lim_{N \to +\infty}\sup_{\theta \in Cl(\mathcal{K}_0)} \norm{S_N(x, \theta)-f(x,\theta)}_2=0$ holds (see \cite{taylor}). Moreover, using \eqref{fk2} in Lemma \ref{fk} and  Lemma \ref{fk3}, for all $x \in E$ we have
\begin{align*}
\norm{S_N(x, \theta)-f(x,\theta)}_2\leq \sum_{\mathbf{k}>N}\norm{C_{\mathbf{k}}(x)}_2\leq k_2  \sum_{\mathbf{k}>N} K_{d+1}(\mathbf{k})<+\infty.
\end{align*}
From the previous inequality we can easily determine that $M_N \to 0$.\\
In the same manner as in \eqref{tdk1}, we obtain
\begin{align*}
\sup_{t \in [0, T]} |V_t^{(N)}-V_t|\leq 2 M_N c_{21} \to 0, \quad N\to +\infty.
\end{align*}
The dominated convergence theorem for stochastic integrals \cite[Theorem 2.12]{ry} applied on It\^o integrals with respect to each component of Brownian motion $(W_t)_t$ yields
\begin{align*}
\int_0^T V_t^{(N)}\,dW_t^l \overset{\mathbb{P}}{\to} \int_0^T V_t\, dW_t^l, \text{ for } l=1,\dots, k,
\end{align*}
when $N \to +\infty$.
Let us denote
\begin{align*}
I_N(\theta)&\coloneqq \sum_{l=1}^k \int_0^T V_t^{(N)}\,dW_t^l=\sum_{i=1}^n \int_{t_{i-1}}^{t_i} \left<S_N(X_t, \theta)-S_N(X_{i-1},\theta)\,\big| \, \nu(X_t)\,dW_t\right>,\\
I(\theta)&\coloneqq \sum_{l=1}^k \int_0^T V_t\,dW_t^l=\sum_{i=1}^n \int_{t_{i-1}}^{t_i} \left<f(X_t,\theta)-f(X_{i-1}, \theta)\,\big| \, \nu(X_t)\,dW_t\right>.
\end{align*}
Using continuous mapping theorem \cite[Theorem 2.3]{vaart}, it follows that $I_N(\theta) \overset{\mathbb{P}}{\to} I(\theta)$ when $N \to +\infty$.\\
Hence, for every $\theta \in \mathcal{K}_0\cap \mathbb{Q}^d$ exists a subsequence $(N_p) \equiv (N_p(\theta))$ and an event $A(\theta)$, $\mathbb{P}(A)=1$, such that for all $\omega \in A(\theta)$, $\lim_{p \to +\infty} I_{N_p}(\theta)(\omega)=I(\theta)(\omega)$.
We define $A_0 \coloneqq \cap_{\theta \in \mathcal{K}_0 \cap \mathbb{Q}^d}A(\theta)$ and it is also the events of probability one. On event $A_0$, for $\theta \in \mathcal{K}_0 \cap \mathbb{Q}^d$ we have that
\begin{align*}
\big|I(\theta)\big|&\leq \big|I(\theta)-I_{N_p}(\theta) \big|+\\
&+\Big|\sum_{i=1}^n \int_{t_{i-1}}^{t_i} \left<\sum_{|\mathbf{k}|\leq N_p}\left(C_{\mathbf{k}}(X_t)-C_{\mathbf{k}}(X_{i-1})\right)e^{-i\left<\mathbf{k}|\theta\right>}\,\Big| \, \nu(X_t)\,dW_t\right>\Big|\leq\\
&\leq \big|I(\theta)-I_{N_p}(\theta) \big|+ \sum_{\mathbf{k}\in \mathbb{Z}^d} \Big|\sum_{i=1}^n \int_{t_{i-1}}^{t_i} \left<C_{\mathbf{k}}(X_t)-C_{\mathbf{k}}(X_{i-1})\,\big| \, \nu(X_t)\,dW_t\right>\Big|.
\end{align*}
After letting $p \to +\infty$, it follows
\begin{align*}
\big|I(\theta)\big|\leq \sum_{\mathbf{k}\in \mathbb{Z}^d} \Big|\sum_{i=1}^n \int_{t_{i-1}}^{t_i}\left<C_{\mathbf{k}}(X_t)-C_{\mathbf{k}}(X_{i-1})\,\big| \, \nu(X_t)\,dW_t\right>\Big|.
\end{align*}
For $g(X_t,\theta)=\sum_{i=1}^n\left(f(X_t,\theta)-f(X_{i-1}, \theta)\right)\mathbbm{1}_{\left<t_{i-1},\, t_i\right]}(t)$ using Theorem \ref{nepr} the mapping $\theta \mapsto I(\theta)$ is almost surely a continuous function. There is an event $B_0$ of probability one such that for every $\omega \in B_0$, $\sup_{\theta \in \mathcal{K}_0}\big|I(\theta)(\omega)\big|=\sup_{\theta \in \mathcal{K}_0\cap \mathbb{Q}^d}\big|I(\theta)(\omega)\big|$. Accordingly, on event $C_0\coloneqq A_0\cap B_0$ of probability one, we have that
\begin{align}\label{red}
\sup_{\theta \in \mathcal{K}_0}\big|I(\theta)\big|\leq \sum_{\mathbf{k}\in \mathbb{Z}^d} \Big|\sum_{i=1}^n \int_{t_{i-1}}^{t_i}\left<C_{\mathbf{k}}(X_t)-C_{\mathbf{k}}(X_{i-1})\,\big| \, \nu(X_t)\,dW_t\right>\Big|.
\end{align}
In analysing the squared $L^2$-norm of the right-hand side in \eqref{red}, we use Lemma \ref{iso} \eqref{viito}, the norm consistency, \eqref{fk1} from Lemma \ref{fk} ($m=d+1$), and Lemma \ref{pl}. The last inequality is obtained in the same way as in \eqref{end}.
\allowdisplaybreaks
\begin{align*}
&\mathbb{E}\left[\Big|\sum_{i=1}^n \int_{t_{i-1}}^{t_i} \left<C_{\mathbf{k}}(X_t)-C_{\mathbf{k}}(X_{i-1})\,\big| \, \nu(X_t)\,dW_t\right>\Big|^2\right]=\\
&=\sum_{i=1}^n \mathbb{E}\left[\left(\int_{t_{i-1}}^{t_i} \left(C_{\mathbf{k}}(X_t)-C_{\mathbf{k}}(X_{i-1})\right)^T \nu(X_t)\,dW_t\right)^2\right]\\
&+2\sum_{1 \leq i<l \leq n}\mathbb{E}\Biggl[\int_{t_{i-1}}^{t_i}\left<C_{\mathbf{k}}(X_t)-C_{\mathbf{k}}(X_{i-1})\,\big|\,\nu(X_t)\,dW_t\right>\cdot\\
&\int_{t_{l-1}}^{t_l}\left<C_{\mathbf{k}}(X_t)-C_{\mathbf{k}}(X_{l-1})\,\big|\,\nu(X_t)\,dW_t\right>\Biggr]=\\
&=\sum_{i=1}^n \mathbb{E}\left[\int_{t_{i-1}}^{t_i} \norm{\nu(X_t)^T\left(C_{\mathbf{k}}(X_t)-C_{\mathbf{k}}(X_{i-1})\right)}_2^2\,dt\right]\leq\\
&\leq c_{22} \mathbb{E}\left[\sum_{i=1}^n \int_{t_{i-1}}^{t_i} \norm{C_{\mathbf{k}}(X_t)-C_{\mathbf{k}}(X_{i-1})}_2^2\,dt\right]\leq\\
&\leq c_{22} \mathbb{E}\left[\sum_{i=1}^n \int_{t_{i-1}}^{t_i}k_1^2 \left(\frac{d+1}{1+|k_1|+\cdots+|k_d|}\right)^{2(d+1)} \norm{X_t-X_{i-1}}_2^2\,dt\right]\leq\\
&\leq c_{23} \left(\frac{d+1}{1+|k_1|+\cdots+|k_d|}\right)^{2(d+1)} \sum_{i=1}^n \int_{t_{i-1}}^{t_i} \mathbb{E}\left[\norm{X_t-X_{i-1}}_2^2\right]\,dt\leq\\
&\leq c_{24} (d+1)^{2(d+1)}\frac{1}{(1+|k_1|+\cdots+|k_d|)^{2(d+1)}} \Delta_n
\end{align*}
Since series $\sum_{\mathbf{k}\in \mathbb{Z}^d} \left(\frac{1}{1+|k_1|+\cdots+|k_d|}\right)^{d+1}$ converges, we have
\begin{align*}
\sum_{k \in \mathbb{Z}^d} \norm{\sum_{i=1}^n \int_{t_{i-1}}^{t_i} \left<C_{\mathbf{k}}(X_t)-C_{\mathbf{k}}(X_{i-1})\,\big| \, \nu(X_t)\,dW_t\right>}_{L^2}< + \infty.
\end{align*}
Therefore, the series on the right hand side of \eqref{red} converges in $L^2$ and almost sure to the same limits that are equal almost sure \cite[Proposition 2.10.1]{brockdavis}.  Finally, we obtain
\begin{align*}
\norm{\sup_{\theta \in \mathcal{K}_0} |I(\theta)|}_{L^2}\leq C_2 \sqrt{\Delta_n}.
\end{align*}
\end{proof}

\begin{proof} [Proof of Theorem \ref{main1}]
Let us assume that $E$ is a compact set. Let $c_8, \dots, c_{19}$ be positive constants occurring in the proof.\\
Using abbreviation in \eqref{gj} we define a random vector $\chi_i^n \coloneqq [ \chi_i^{n,1}, \dots, \chi_i^{n,d}]^T$ where each component is equal to
\begin{align*}
\chi_i^{n,j} \coloneqq \frac{1}{\sqrt{\Delta_n}} \int_{t_{i-1}}^{t_i}\left<\int_{t_{i-1}}^s \nabla g_j(X_u)\nu(X_u) \, dW_u\,\bigg| \, \nu(X_s) \, dW_s\right>.
\end{align*}
We also introduce matrix functions $R^{(j)}$ and column vectors \\ $J^{(i,j)}=\left[J^{(i,j)}_1, \dots, J^{(i,j)}_k\right]^T$ in the following form.
\begin{align*}
R^{(j)}(t, \theta)\coloneqq \nabla g_j(X_t)\nu(X_t),\quad j=1,2,\dots, d
\end{align*}
\begin{align*}
J^{(i,j)}(s)\coloneqq \int_{t_{i-1}}^s R^{(j)}(u, \theta)\, dW_u
\end{align*}
Let $C=(C_{jl}(t))_{0\leq t \leq T}$ be a continuous adapted process given by
\begin{align*}
C_{jl}(t)=\frac{1}{2}\int_0^t \sum_{p,r=1}^k S_{pr}(X_s)\left<e_r^T R^{(j)}(s,\theta)\,| \, e_p^T R^{(l)}(s,\theta)\right>\,ds, \quad j,l=1, \dots, d.
\end{align*}
We will prove that all conditions of Theorem \ref{jac} are fulfilled.
Since $E$ is a compact set, the matrix function $R^{(j)}(t, \theta)$ is bounded for every $j=1,2, \dots, d$, and  for every $i=1,2,\dots, n$ vectors $\chi_i^{n}$ are square-integrable random vectors. Using the definition of the scalar product, it follows that for all $i=1,2,\dots, n$,  $\chi_i^{n}$ is a sum of It\^o integrals, hence it is a martingale. The equality $\mathbb{E}\left[\chi_i^{n,j}|\mathcal{F}_{n, i-1}\right]=0$ trivially implies the statement of condition \eqref{ppt1}.

To satisfy condition \eqref{ppt3}, it is sufficient to prove that for $1 \leq j \leq d$ and $1\leq m \leq k$
\begin{align*}
\sum_{i=1}^{A_n^t} \mathbb{E}\left[\chi_i^{n,j} \left(W_{t_i}^m - W_{t_{i-1}}^m\right)\bigg| \mathcal{F}_{n,i-1}\right] \overset{\mathbb{P}}{\to} 0.
\end{align*}
Using the product formula, we have
\begin{align*}
&\sum_{i=1}^{A_n^t} \mathbb{E}\left[\chi_i^{n,j} \left(W_{t_i}^m - W_{t_{i-1}}^m\right)\big| \mathcal{F}_{n,i-1}\right]=\\
&=\frac{1}{\sqrt{\Delta_n}}\sum_{i=1}^{A_n^t} \sum_{p=1}^k \mathbb{E}\left[\int_{t_{i-1}}^{t_i} \left(J^{(i,j)}(s)\right)^T \nu(X_s) e_p \,dW_s^p \cdot \int_{t_{i-1}}^{t_i} \,dW_s^m \bigg| \mathcal{F}_{n, i-1}\right]=\\
&= \frac{1}{\sqrt{\Delta_n}}\sum_{i=1}^{A_n^t} \sum_{p=1}^k \mathbb{E}\Bigl[ \int_{t_{i-1}}^{t_i} \int_{t_{i-1}}^s\left(J^{(i,j)}(u)\right)^T \nu(X_u) e_p \,dW_u^p \,dW_s^m+\\
&+\int_{t_{i-1}}^{t_i} \int_{t_{i-1}}^s \,dW_u^m \left(J^{(i,j)}(s)\right)^T \nu(X_s) e_p \,dW_s^p +\\
&+ \left<\int_{t_{i-1}}^{t_i} \left(J^{(i,j)}(s)\right)^T \nu(X_s) e_p \,dW_s^p, \, \int_{t_{i-1}}^{t_i}\,dW_s^m\right> \bigg|\mathcal{F}_{n, i-1}\Bigr].
\end{align*}
Using the martingale property and the independence of the components of Brownian motion, we have
\begin{align*}
&\sum_{i=1}^{A_n^t} \mathbb{E}\left[\chi_i^{n,j} \left(W_{t_i}^m - W_{t_{i-1}}^m\right)\big| \mathcal{F}_{n,i-1}\right]=\\
&=\frac{1}{\sqrt{\Delta_n}}\sum_{i=1}^{A_n^t} \mathbb{E}\left[ \int_{t_{i-1}}^{t_i} \left(J^{(i,j)}(s)\right)^T \nu(X_s) e_m\,ds \big| \mathcal{F}_{n, i-1}\right]=\\
&=\frac{1}{\sqrt{\Delta_n}}\sum_{i=1}^{A_n^t} \mathbb{E}\left[\int_{t_{i-1}}^{t_i} \left(J^{(i,j)}(s)\right)^T\left(\nu(X_s)-\nu(X_{i-1})\right) e_m\,ds \big|\mathcal{F}_{n, i-1}\right]+\\
&+\frac{1}{\sqrt{\Delta_n}}\sum_{i=1}^{A_n^t} \mathbb{E}\left[ \int_{t_{i-1}}^{t_i} \left(J^{(i,j)}(s)\right)^T \nu(X_{i-1}) e_m \, ds\big|\mathcal{F}_{n,i-1}\right].
\end{align*}
We denote by $\tilde{\nu}_m(x)$ the $m$-th column of matrix $\nu(x)$. Using Lemma \ref{mito} for function $\tilde{\nu}_m$ on interval $[t_{i-1}, s]$, we have
\begin{align}
&\sum_{i=1}^{A_n^t} \mathbb{E}\left[\chi_i^{n,j} \left(W_{t_i}^m - W_{t_{i-1}}^m\right)\big| \mathcal{F}_{n,i-1}\right]=\notag\\
&=\frac{1}{\sqrt{\Delta_n}}\sum_{i=1}^{A_n^t} \mathbb{E}\left[\left(J^{(i,j)}(s)\right)^T \left(\tilde{\nu}_m(X_s)-\tilde{\nu}_m (X_{i-1})\right) \,ds \big| \mathcal{F}_{n, i-1}\right]=\notag\\
&=\frac{1}{\sqrt{\Delta_n}}\sum_{i=1}^{A_n^t} \mathbb{E}\Bigl[\int_{t_{i-1}}^{t_i} \left(J^{(i,j)}(s)\right)^T\int_{t_{i-1}}^s \Bigl(\nabla \tilde{\nu}_m(X_u)\mu(X_u, \theta_0)+\notag\\
&+\frac{1}{2}\nabla_2 \tilde{\nu}_m(X_u)\Bigr)\,du\,ds \big|\mathcal{F}_{n, i-1}\Bigr]+ \label{48a1}\\
&+\frac{1}{\Delta_n}\sum_{i=1}^{A_n^t} \mathbb{E}\left[\int_{t_{i-1}}^{t_i} \left(J^{(i,j)}(s)\right)^T\int_{t_{i-1}}^s \nabla \tilde{\nu}_m(X_u)\nu(X_u)\,dW_u\,ds \big| \mathcal{F}_{n, i-1}\right].\label{48a2}
\end{align}
Next, we analyze the $L^1$-norm of \eqref{48a1}. There exists a constant $c_{8}$ such that
\begin{align*}
\mathbb{E}\Biggl[\Bigg|\frac{1}{\sqrt{\Delta_n}}\sum_{i=1}^{A_n^t} \mathbb{E}\Biggl[\int_{t_{i-1}}^{t_i}&\Bigl<J^{(i,j)}(s)\,\bigg| \,\int_{t_{i-1}}^s \Bigl(\nabla \tilde{\nu}_m(X_u)\mu(X_u, \theta_0)+\\
&+\frac{1}{2}\nabla_2 \tilde{\nu}_m(X_u)\Bigr)\,du\Bigr>\,ds\big|\mathcal{F}_{n,i-1}\Biggr]\Bigg|\Biggr]\leq \\
&\leq c_{8} \sqrt{\Delta_n}T \mathbb{E}[M],
\end{align*}
where $M  \coloneqq \sup_{0 \leq s \leq T} \norm{\int_0^s R^{(j)}(u, \theta)\,dW_u}_2$. Using Doob's maximal inequality for vector martingale, we have that \eqref{48a1} converges to 0 in $L^1$ norm, so it also converges to 0 in probability.\\
Let us denote $H_m(x)\coloneqq\nabla \tilde{\nu}_m(x)\nu(x)$. Using It\^o formula for function $F(x,y)=\left<x\,|\,y\right>$, we have
\allowdisplaybreaks
\begin{align*}
&\frac{1}{\Delta_n}\sum_{i=1}^{A_n^t} \mathbb{E}\left[\int_{t_{i-1}}^{t_i}\left< \int_{t_{i-1}}^s R^{(j)}(u, \theta)\,dW_u\,\bigg| \,\int_{t_{i-1}}^s H_m(X_u)\,dW_u\right>\,ds \big| \mathcal{F}_{n, i-1}\right]=\\
&=\frac{1}{\sqrt{\Delta_n}}\sum_{i=1}^{A_n^t} \mathbb{E}\Bigl[\int_{t_{i-1}}^{t_i}\int_{t_{i-1}}^s \left<\int_{t_{i-1}}^u H_m(X_v)\,dW_v\,\bigg| \, R^{(j)}(u, \theta)\,dW_u\right>\,ds \big| \mathcal{F}_{n, i-1}\Bigr]+\\
&+\frac{1}{\sqrt{\Delta_n}}\sum_{i=1}^{A_n^t} \mathbb{E}\Bigl[\int_{t_{i-1}}^{t_i} \int_{t_{i-1}}^s \left<\int_{t_{i-1}}^u R^{(j)}(v, \theta)\,dW_v\,\bigg| \, H_m(X_u)\,dW_u\right>\,ds\big| \mathcal{F}_{n, i-1}\Bigr]+\\
&+\frac{1}{\sqrt{\Delta_n}}\sum_{i=1}^{A_n^t}\sum_{p,r=1}^k  \mathbb{E}\Bigl[\int_{t_{i-1}}^{t_i} \int_{t_{i-1}}^s \left(R^{(j)}(u, \theta) \circ H_m(X_u)\right)_{pr}\,du\,ds\big|\mathcal{F}_{n, i-1}\Bigr].
\end{align*}
Using the well-known characterization of $L^1$ random variables, it is easy to prove that the integral sign and expectation can be interchanged so that the first two summands are zero. For the third one there is a constant $c_{9}$ such that
\begin{align*}
&\Bigg| \frac{1}{\sqrt{\Delta_n}}\sum_{i=1}^{A_n^t}\sum_{p,r=1}^k \mathbb{E}\Bigl[\int_{t_{i-1}}^{t_i} \int_{t_{i-1}}^s \left(R^{(j)}(u, \theta) \circ H_m(X_u)\right)_{pr}\,du\,ds\big|\mathcal{F}_{n, i-1}\Bigr]\Bigg|\leq\\
&\leq c_{9}T\sqrt{\Delta_n}.
\end{align*}
We conclude that expression \eqref{48a2} converges almost surely to 0, so it also converges to zero in probability. 

To prove \eqref{ppt5}, it is enough to show that for arbitrary $j$,  $ 1 \leq j, \leq d$
\begin{align*}
\sum_{i=1}^{A_n^t} \mathbb{E}\left[ \chi_i^{n,j} \left(N_{t_i}-N_{t_{i-1}}\right)\big|\mathcal{F}_{n,i-1}\right] \overset{\mathbb{P}}{\to} 0, \quad \forall N \in \mathcal{M}_b \left(W^{\perp}\right).
\end{align*}
In a similar way as before, we have
\allowdisplaybreaks
\begin{align*}
&\sum_{i=1}^{A_n^t} \mathbb{E}\left[ \chi_i^{n,j} \left(N_{t_i}-N_{t_{i-1}}\right)\big|\mathcal{F}_{n,i-1}\right]=\\
&=\frac{1}{\sqrt{\Delta_n}} \sum_{i=1}^{A_n^t} \mathbb{E}\left[\int_{t_{i-1}}^{t_i} \left<J^{(i,j)}(s, \theta)\,\bigg| \, \nu(X_s)\, dW_s\right>\cdot \int_{t_{i-1}}^{t_i}\, dN_s \big|\mathcal{F}_{n, i-1}\right]=\\
&=\frac{1}{\sqrt{\Delta_n}} \sum_{i=1}^{A_n^t} \sum_{p=1}^k \mathbb{E}\left[\int_{t_{i-1}}^{t_i} \int_{t_{i-1}}^s \left<J^{(i,j)}(u, \theta)\,\bigg| \,\nu(X_u) e_p\right>\,dW_u^p\,dN_s\big| \mathcal{F}_{n,i-1}\right]+\\
&+\frac{1}{\sqrt{\Delta_n}} \sum_{i=1}^{A_n^t} \sum_{p=1}^k \mathbb{E}\left[\int_{t_{i-1}}^{t_i} \int_{t_{i-1}}^s \,dN_u \left<J^{(i,j)}(s,\theta)\,\bigg| \,\nu(X_s) e_p\right>\,dW_s^p\big|\mathcal{F}_{n,i-1}\right]+\\
&+\frac{1}{\sqrt{\Delta_n}} \sum_{i=1}^{A_n^t} \sum_{p=1}^k \mathbb{E}\left[\left<\int_{t_{i-1}}^{t_i}\left<J^{(i,j)}(u,\theta)\bigg| \, \nu(X_u) e_p\right>\,dW_u^p, \int_{t_{i-1}}^{t_i}dN_u\right>\big| \mathcal{F}_{n, i-1} \right]=\\
&=\frac{1}{\sqrt{\Delta_n}} \sum_{i=1}^{A_n^t} \sum_{p=1}^k \mathbb{E}\left[\int_{t_{i-1}}^{t_i}\left<J^{(i,j)}(u,\theta)\,\big| \, \nu(X_u) e_p\right>\,d\left[W^p, N\right]_u\,\big| \mathcal{F}_{n, i-1} \right]=\\
&=0.
\end{align*}
We use the fact that the conditional expectation of the martingale difference is zero \cite[Theorem 2.2]{ry}. In the last equality we use the orthogonality of the process $N$ on the components of Brownian motion.

Next, we prove \eqref{ppt4}. Let $\epsilon>0$. Using Cauchy-Schwarz and Markov inequality  \cite[Theorem 1.6.4]{dur}, we have
\begin{align*}
\mathbb{E}\Big|\sum_{i=1}^n \mathbb{E}\left[\norm{\chi_i^n}_2^2 \cdot \mathbbm{1}_{\{\norm{\chi_i^n}_2> \epsilon\}}\big| \mathcal{F}_{n, i-1} \right]  \Big|
&= \sum_{i=1}^n \mathbb{E}\left[\norm{\chi_i^n}_2^4\right]^{\frac{1}{2}} \cdot \mathbb{P}\left(\norm{\chi_i^n}_2>\epsilon\right)^{\frac{1}{2}}\leq\\
&\leq \sum_{i=1}^n \mathbb{E}\left[\norm{\chi_i^n}_2^4\right]^{\frac{1}{2}} \left(\frac{\mathbb{E}\left[ \norm{\chi_i^n}_2^4\right]}{\epsilon^4}\right)^{\frac{1}{2}}=\\
&=\frac{1}{\epsilon^2}\sum_{i=1}^n \mathbb{E}\left[\norm{\chi_i^n}_2^4\right]=\\
&\leq \frac{d}{\epsilon^2}\sum_{i=1}^n  \mathbb{E}\left[\sum_{j=1}^d (\chi_i^{n,j})^4\right].
\end{align*}
To satisfy condition \eqref{ppt4}, it is sufficient to prove that
\begin{align*}
\lim_{n \rightarrow +\infty} \frac{1}{\epsilon^2} \sum_{i=1}^n \mathbb{E}\left[\left(\chi_i^{n,j}\right)^4\right]=0.
\end{align*}
If we denote by
\begin{align*}
H^{(i,j)}(s)\coloneqq \int_{t_{i-1}}^s \left<J^{(i,j)}(u)\,\bigg| \, \nu(X_u)\, dW_u\right>,
\end{align*}
there exist a constant $c_{10}$ such that
\begin{align}
&\frac{1}{\epsilon^2}\sum_{i=1}^n \mathbb{E}\left[\left(\chi_i^{n,j}\right)^4\right]=\notag\\
&=\frac{1}{\Delta_n^2 \epsilon^2}\sum_{i=1}^n \mathbb{E}\left[\left(H^{(i,j)}(t_i)\right)^4\right]=\notag\\
&=\frac{1}{\Delta^2\epsilon^2}\sum_{i=1}^n \mathbb{E}\left[4\int_{t_{i-1}}^{t_i}\left(H^{(i,j)}(s)\right)^3 \left(J^{(i,j)}(s)\right)^T \nu(X_s)\,dW_s\right]+\notag\\
&+\frac{1}{\Delta_n^2 \epsilon^2}\sum_{i=1}^n \mathbb{E}\left[6\int_{t_{i-1}}^{t_i}\left(H^{(i,j)}(s)\right)^2\,d\left<H^{(i,j)}\right>_s\right]=\notag\\
&=\frac{6}{\Delta_n^2\epsilon^2}\sum_{i=1}^n \sum_{p=1}^k \mathbb{E}\left[\int_{t_{i-1}}^{t_i} \left(H^{(i,j)}(s)\right)^2\left<J^{(i,j)}(s)\,| \, \nu(X_s) e_p\right>^2\,ds\right]\leq\notag\\
&\leq \frac{c_{10} }{\Delta_n^2 \epsilon^2} \sum_{i=1}^n \sum_{r=1}^k \mathbb{E}\left[\int_{t_{i-1}}^{t_i} \left(H^{(i,j)}(s)\right)^2 \cdot \left(J_r^{(i,j)}(s)\right)^2\,ds\right]\label{49}.
\end{align}
To analyze the expression under the integral sign, we apply It\^o formula for the function $F(x,y)=x^2y^2$. The independence of the components of Brownian motion is essential for the calculation of the quadratic covariance and variances.
\begin{align}
&\left(H^{(i,j)}\right)^2 \cdot \left(J_r^{(i,j)}\right)^2=\notag\\
&=2\int_{t_{i-1}}^s H^{(i,j)}(u) \left(J_r^{(i,j)}(u)\right)^2\,dH^{(i,j)}(u)+\notag\\
&+2\int_{t_{i-1}}^s \left(H^{(i,j)}(u)\right)^2 J_r^{(i,j)}(u)\,dJ_r^{(i,j)}(u)+\notag\\
&+\frac{1}{2}\int_{t_{i-1}}^s 2 \left(J_r^{(i,j)}(u)\right)^2 \, d\left<H^{(i,j)}, \, H^{(i,j)}\right>_u +\notag\\
&+\frac{1}{2}\cdot 2\int_{t_{i-1}}^s 4H^{(i,j)}(u) J_r^{(i,j)}(u)\,d\left<H^{(i,j)}, \, J_r^{(i,j)}\right>_u+\notag\\
&+\frac{1}{2}\int_{t_{i-1}}^s 2 \left(H^{(i,j)}\right)^2\,d\left<J_r^{(i,j)},\, J_r^{(i,j)}\right>_u=\notag\\
&=2\int_{t_{i-1}}^s H^{(i,j)}(u) \left(J_r^{(i,j)}(u)\right)^2\left(J^{(i,j)}(u)\right)^T\nu(X_u)\, dW_u+\notag\\
&+2\int_{t_{i-1}}^s \left(H^{(i,j)}(u)\right)^2 J_r^{(i,j)}(u)e_r^T R^{(i,j)}(u, \theta)\,dW_u+\notag\\
&+\int_{t_{i-1}}^s \left(J_r^{(i,j)}(u)\right)^2 \norm{\nu(X_u)^T J^{(i,j)}(u)}_2^2\,du+\notag\\
&+4\int_{t_{i-1}}^s H^{(i,j)}(u) J_r^{(i,j)}(u)\left<\left(J^{(i,j)}(u)\right)^T\nu(X_u)\,\bigg| \, e_r^T R^{(j)}(u, \theta)\right>\,du+\notag\\
&+\int_{t_{i-1}}^s \left(H^{(i,j)}(u)\right)^2 \norm{e_r^T R^{(j)}(u, \theta)}_2^2\,du\label{hj1}
\end{align}
Considering equality \eqref{hj1}, we conclude that
\begin{align*}
&\frac{1}{\epsilon^2}\sum_{i=1}^n \mathbb{E}\left[\left(\chi_i^{n,j}\right)^4\right]\leq \\
&\leq \frac{12 c_{10}}{\Delta_n^2 \epsilon^2}\sum_{i=1}^n \sum_{r=1}^k \mathbb{E}\left[\int_{t_{i-1}}^{t_i} \int_{t_{i-1}}^s \left(J_r^{(i,j)}(u)\right)^2 \norm{\nu(X_u)^T J^{(i,j)}(u)}_2^2\,du\,ds\right]+\\
&+\frac{48 c_{10}}{\Delta_n^2 \epsilon^2}\sum_{i=1}^n \sum_{r=1}^k \mathbb{E}\Bigl[\int_{t_{i-1}}^{t_i}\int_{t_{i-1}}^s H^{(i,j)}(u)J_r^{(i,j)}(u)\cdot \\
&\cdot \left<\left(J^{(i,j)}(u)\right)^T\nu(X_u)\,\bigg| \, e_r^T R^{(j)}(u, \theta)\right>\, du\, ds\Bigr] +\\
&+\frac{12 c_{10}}{\Delta_n^2 \epsilon^2}\sum_{i=1}^n \sum_{r=1}^k \mathbb{E}\left[\int_{t_{i-1}}^{t_i} \int_{t_{i-1}}^s \left(H^{(i,j)}(u)\right)^2 \norm{e_r^T R^{(j)}(u, \theta)}_2^2\,du\,ds\right].
\end{align*}
By a similar reasoning as before (using norm consistency, It\^o formula, Lemma \ref{iso} (\ref{miito}) and (\ref{viito})), there exist constants $c_{12}$ and $c_{14}$ such that
\begin{align*}
&\frac{12 c_{10}}{\Delta_n^2 \epsilon^2}\sum_{i=1}^n \sum_{r=1}^k \mathbb{E}\left[\int_{t_{i-1}}^{t_i} \int_{t_{i-1}}^s \left(J_r^{(i,j)}(u)\right)^2 \norm{\nu(X_u)^TJ^{(i,j)}(u)}_2^2\,du\,ds\right]\leq\\
&\leq \frac{12c_{11}}{\Delta_n^2 \epsilon^2}\sum_{i=1}^n \sum_{p,r=1}^k \int_{t_{i-1}}^{t_i} \int_{t_{i-1}}^s \int_{t_{i-1}}^u \Bigl[ \mathbb{E}\left[\left(J_p^{(i,j)}(v)\right)^2\right]+4 \mathbb{E}\left[J_p^{(i,j)}(v)J_r^{(i,j)}(v)\right]+\\&+\mathbb{E}\left[\left(J_r^{(i,j)}(v)\right)^2\right]\Bigr]\,dv\,du\,ds\leq c_{12}\Delta_n T,
\end{align*}
and
\begin{align*}
&\frac{12c_{10}}{\Delta_n^2\epsilon^2}\sum_{i=1}^n \sum_{r=1}^k \int_{t_{i-1}}^{t_i} \int_{t_{i-1}}^s \mathbb{E}\left[\left(H^{(i,j)}(u)\right)^2 \norm{e_r^T R^{(j)}(u, \theta)}_2^2\right]\,du\,ds \leq\\
&\leq \frac{c_{13}}{\Delta_n^2\epsilon^2} \sum_{i=1}^n \int_{t_{i-1}}^{t_i} \int_{t_{i-1}}^s \mathbb{E}\left[\left(H^{(i,j)}(u)\right)^2\right]\,du\,ds\leq c_{14}\Delta_nT.
\end{align*}
Using the previous two bounds, there is a constant $c_{16}$ such that
\begin{alignat*}{3}
&\Big|\frac{48 c_{10}}{\Delta_n^2 \epsilon^2}\sum_{i=1}^n \sum_{r=1}^k \mathbb{E}\Bigl[\int_{t_{i-1}}^{t_i}\int_{t_{i-1}}^s && H^{(i,j)}(u) J_r^{(i,j)}(u)\cdot \\
& &&\cdot \left<\left(J^{(i,j)}(u)\right)^T\nu(X_u)\,\bigg| \, e_r^T R^{(j)}(u, \theta)\right>\, du\, ds\Bigr]\Big|\leq &&\\
&\leq \frac{c_{15}}{\Delta_n^2\epsilon^2} \sum_{i=1}^{n} \sum_{r=1}^k \int_{t_{i-1}}^{t_i} \int_{t_{i-1}}^s \mathbb{E}\Bigl[&&\left(J_r^{(i,j)}(u)\right)^2  \norm{\left(J^{(i,j)}(u)\right)^T\nu(X_u)}_2^2\Bigr]^{\frac{1}{2}} \cdot \\
& &&\cdot \mathbb{E}\Bigl[ \left(H^{(i,j)}\right)^2 \norm{e_r^T R^{(j)}(u, \theta)}_2^2\Bigr]^{\frac{1}{2}}\,du\,ds \leq &&\\
&\leq c_{16}\Delta_n T. && &&
\end{alignat*}
Hence, expression in \eqref{49} is bounded above by $c_{17} \Delta_n T$ so it converges to zero when $n \to +\infty$.

Lastly, we show that assumption \eqref{ppt2} is fulfilled. For arbitrary and fixed $j , l$ such that $1 \leq j \leq l \leq d$ we consider
\begin{align*}
&\sum_{i=1}^{A_n^t} \left(\mathbb{E}\left[\chi_i^{n,j}\chi_i^{n,l}\big|\mathcal{F}_{n,i-1}\right]-\mathbb{E}\left[\chi_i^{n,j}\big|\mathcal{F}_{n,i-1}\right]\cdot \mathbb{E}\left[\chi_i^{n,l}\big|\mathcal{F}_{n,i-1}\right]\right)=\\
&=\frac{1}{\Delta_n}\sum_{i=1}^{A_n^t}\sum_{p,r=1}^k \mathbb{E}\Bigl[\int_{t_{i-1}}^{t_i} \left<J^{(i,j)}(s)\,| \, \nu(X_s) e_p\right>\, dW_s^p \cdot \\
&\cdot \int_{t_{i-1}}^{t_i} \left<J^{(i,l)}(s)\,| \, \nu(X_s) e_r\right>\, dW_s^r\big|\mathcal{F}_{n, i-1}\Bigr].
\end{align*}
Using the product formula and independence of the components of Brownian motion, we have
\begin{align}
&\sum_{i=1}^{A_n^t} \left(\mathbb{E}\left[\chi_i^{n,j}\chi_i^{n,l}\big|\mathcal{F}_{n,i-1}\right]-\mathbb{E}\left[\chi_i^{n,j}\big|\mathcal{F}_{n,i-1}\right]\cdot \mathbb{E}\left[\chi_i^{n,l}\big|\mathcal{F}_{n,i-1}\right]\right)=\notag\\
&=\frac{1}{\Delta_n}\sum_{i=1}^{A_n^t}\sum_{p=1}^k \mathbb{E}\left[\int_{t_{i-1}}^{t_i} \left<J^{(i,j)}(s)\,| \,\nu(X_s) e_p\right>\cdot \left<J^{(i,l)}(s)\,| \,\nu(X_s) e_p\right>\,ds\big|\mathcal{F}_{n, i-1}\right]\label{last}.
\end{align}
Using simple relations of linear algebra it results that \eqref{last} is equal to
\begin{align*}
\frac{1}{\Delta_n}\sum_{i=1}^{A_n^t}\sum_{p,r=1}^k\mathbb{E}\left[\int_{t_{i-1}}^{t_i}S_{pr}(X_s)J_r^{(i,j)}(s)J_p^{(i,l)}(s)\,ds \big|\mathcal{F}_{n, i-1}\right].
\end{align*}
Again, the main tool is It\^o formula and independence of the components of Brownian motion.
\begin{align}
&\frac{1}{\Delta_n}\sum_{i=1}^{A_n^t}\sum_{p,r=1}^k\mathbb{E}\left[\int_{t_{i-1}}^{t_i}S_{pr}(X_s)J_r^{(i,j)}(s)J_p^{(i,l)}(s)\,ds \big|\mathcal{F}_{n, i-1}\right]=\notag\\
&=\frac{1}{\Delta_n}\sum_{i=1}^{A_n^t} \sum_{p,r=1}^k \mathbb{E}\left[\int_{t_{i-1}}^{t_i}S_{pr}(X_s)\int_{t_{i-1}}^s J_r^{(i,j)}(u)e_p^TR^{(l)}(u,\theta)\,dW_u\,ds\big|\mathcal{F}_{n, i-1}\right] +\label{b1}\\
&+\frac{1}{\Delta_n}\sum_{i=1}^{A_n^t} \sum_{p,r=1}^k \mathbb{E}\left[\int_{t_{i-1}}^{t_i}S_{pr}(X_s)\int_{t_{i-1}}^s J_p^{(i,l)}(u)e_r^TR^{(j)}(u,\theta)\,dW_u\,ds\big|\mathcal{F}_{n, i-1}\right]+\label{b2}\\
&+\frac{1}{\Delta_n}\sum_{i=1}^{A_n^t} \sum_{p,r=1}^k \mathbb{E}\left[\int_{t_{i-1}}^{t_i}S_{pr}(X_s)\int_{t_{i-1}}^s \left<e_r^TR^{(j)}(u,\theta)| e_p^TR^{(l)}(u,\theta)\right>du ds\big|\mathcal{F}_{n, i-1}\right]\label{b3}
\end{align}
Since the matrix $S$ is symmetric, it is enough to consider one of the terms \eqref{b1} and \eqref{b2}. Let $\tilde{h}_{r,p}^{(i,j,l)}(u,\theta)\coloneqq J_r^{(i,j)}(u)e_p^T R^{(l)}(u,\theta)$. Using Jensen's inequality for conditional expectation, Cauchy-Schwarz inequality and some simple inequalities, and Lemma \ref{vito}, it can be shown that there is a constant $c_{18}$ such that
\begin{align*}
&\mathbb{E}\left[\left(\frac{1}{\Delta_n}\sum_{i=1}^{A_n^t}\sum_{p,r=1}^k \mathbb{E}\left[\int_{t_{i-1}}^{t_i}S_{pr}(X_s)\int_{t_{i-1}}^s \tilde{h}_{r,p}^{(i,j,l)}(u, \theta)\,dW_u\,ds \big| \mathcal{F}_{n, i-1}\right]\right)^2\right]=\\
&=\frac{1}{\Delta_n^2}\mathbb{E}\left[\sum_{i=1}^{A_n^t} \left( \mathbb{E}\left[\sum_{p,r=1}^k\int_{t_{i-1}}^{t_i}S_{pr}(X_s)\int_{t_{i-1}}^s \tilde{h}_{r,p}^{(i,j,l)}(u, \theta)\,dW_u\,ds\big|\mathcal{F}_{n, i-1}\right]\right)^2\right]+\\
&+\frac{1}{\Delta_n^2}\mathbb{E}\Bigl[\sum_{i,m=1, i \neq m}^{A_n^t}\left(\sum_{p,r=1}^k \mathbb{E}\left[\int_{t_{i-1}}^{t_i} S_{pr}(X_s)\int_{t_{i-1}}^s \tilde{h}_{r,p}^{(i,j,l)}(u, \theta)\,dWu\,ds \big|\mathcal{F}_{n, i-1}\right]\right)\cdot\\
&\cdot \left(\sum_{p,r=1}^k \mathbb{E}\left[\int_{t_{m-1}}^{t_m}S_{pr}(X_s)\int_{t_{m-1}}^s \tilde{h}_{r,p}^{(m,j,l)}(u, \theta)\,dW_u\,ds \big|\mathcal{F}_{n, m-1}\right]\right)\Bigr]\leq \\
&\leq c_{18}(T\Delta_n + T^2 \Delta_n),
\end{align*}
so \eqref{b1} converges to zero in $L^2$. Hence, it converges to zero in probability, too.\\
To analyze \eqref{b3} we introduce the following notation
\begin{align*}
D^{(i,j,l)}_t \coloneqq \sum_{p,r=1}^k \int_{t_{i-1}}^{t} S_{pr}(X_s)\int_{t_{i-1}}^s \left<e_r^T R^{(j)}(u, \theta)\,|\,e_p^T R^{(l)}(u,\theta)\right>\,du\,ds.
\end{align*}
Then, we have that \eqref{b3} equals
\begin{align}
\frac{1}{\Delta_n}\sum_{i=1}^{A_n^t} \mathbb{E}\left[D^{(i,j,l)}_{t_i}\big| \mathcal{F}_{n, i-1}\right]&=\frac{1}{\Delta_n}\sum_{i=1}^{A_n^t} \left(\mathbb{E}\left[D^{(i,j,l)}_{t_i}\big|\mathcal{F}_{n, i-1}\right]-D^{(i,j,l)}_{t_i}\right)\label{b3a}\\
&+\frac{1}{\Delta_n}\sum_{i=1}^{A_n^t} D^{(i,j,l)}_{t_i}.\label{b3b}
\end{align}
It can be shown that \eqref{b3a} converges to 0 in $L^2$ because there exists a constant $c_{19}$ such that
\begin{align*}
&\mathbb{E}\left[\left(\frac{1}{\Delta_n}\sum_{i=1}^{A_n^t}\left(\mathbb{E}\left[D^{(i,j,l)}_{t_i}|\mathcal{F}_{n,i-1}\right]-D^{(i,j,l)}_{t_i}\right)\right)^2\right]\leq c_{19} T \Delta_n.
\end{align*}
There is only left to show that \eqref{b3b} converges in probability to
\begin{align*}
C_{jl}(t)=\frac{1}{2}\sum_{p,r=1}^k \int_0^t S_{pr}(X_s) \left<e_r^T R^{(j)}(s, \theta)\,| \,e_p^T R^{(l)}(s, \theta)\right>\,ds.
\end{align*}
It is sufficient to show that for arbitrary  $p,r=1,2, \dots, k$ the following convergence holds
\begin{align*}
\frac{1}{\Delta_n}\sum_{i=1}^{A_n^t} \int_{t_{i-1}}^{t_i} S_{pr}(X_s)\int_{t_{i-1}}^s \left<e_r^T R^{(j)}(u, \theta)\,| \, e_p^T R^{(l)}(u, \theta)\right>\,du\,ds \overset{\mathbb{P}}{\to} \\\frac{1}{2}\int_0^t S_{pr}(X_s) \left<e_r^T R^{(j)}(s, \theta)\,|\,e_p^T R^{(l)}(s,\theta)\right>\,ds,
\end{align*}
when $n \to +\infty$.\\
Let us denote $m(u,s)(\omega)\coloneqq S_{pr}(X_s(\omega))\left<e_r^T R^{(j)}(u,\theta)\,| \, e_p^T R^{(l)}(u, \theta)\right>$. For fixed $\omega \in \Omega$, using assumptions (A3), (A4) and compactness of $E$ function  $m(u,s)(\omega)$ is bounded and continuous on $[0,T]\times [0,T]$. Hence, there exist $u_i^{*}(\omega),s_i^{*}(\omega) \in [t_{i-1}, t_i]$ such that $u_i^{*}(\omega)\leq s_i^{*}(\omega)$ and
\begin{align*}
\int_{t_{i-1}}^{t_i} \int_{t_{i-1}}^s m(u,s)\,du\,ds= \frac{\Delta_n^2}{2} m(u_i^{*}, s_i^{*}), \quad i=1,2,\dots, n.
\end{align*}
We have
\begin{align}
&\frac{1}{\Delta_n}\sum_{i=1}^{A_n^t}\int_{t_{i-1}}^{t_i}\int_{t_{i-1}}^s m(u,s)\,du\,ds-\frac{1}{2}\int_0^t m(s,s)\,ds=\notag\\
&=\frac{\Delta_n}{2} \sum_{i=1}^{A_n^t} m(u_i^{*}, u_i^{*})-\frac{1}{2}\int_0^t m(s,s)\,ds+\label{b3b1}\\
&+\frac{\Delta_n}{2}\sum_{i=1}^{A_n^t} \left(m(u_i^{*},s_i^{*})-m(u_i^{*}, u_i^{*})\right).\notag
\end{align}
Since function given by $(u,s) \mapsto m(u,s)$ is continuous, there exists $t^{*} \in \left[t_{A_n^t}, t\right]$ such that $\frac{1}{2}\int_{t_{A_n^t}}^t m(s,s)\,ds=\frac{1}{2}\left(t-t_{A_n^t}\right) m(t^{*}, t^{*})$. Hence,
\begin{align}
&\frac{\Delta_n}{2}\sum_{i=1}^{A_n^t} m(u_i^{*}, u_i^{*})-\frac{1}{2}\int_0^t m(s,s)\,ds=\notag\\
&=\frac{\Delta_n}{2}\sum_{i=1}^{A_n^t} m(u_i^{*}, u_i^{*})+\frac{1}{2}\left(t-t_{A_n^t}\right)m(t^*, t^*)-\frac{1}{2}\int_0^t m(s,s)\,ds-\notag\\
&- \frac{1}{2}\int_{t_{A_n^t}}^t m(s,s)\,ds.\label{m4}
\end{align}
The first two terms of \eqref{m4} forms Riemann integral sum so they converge almost surely to $\frac{1}{2}\int_0^t m(s,s)\,ds$. Boundedness of function $m$ assures that  $\frac{1}{2}\int_{A_n^t}^t m(s,s)\,ds$ converges almost surely to zero so \eqref{b3b1} converges almost surely to zero. \\
Let $\epsilon >0$. The function given by $t \mapsto S_{pr}(X_t)$ is continuous on $[0,T]$. Moreover, it is uniformly continuous so there exists $\delta>0$ such that for all $s,t \in [0,T]$, $|s-t|<\delta$ implies that $|S_{pr}(X_s)-S_{pr}(X_t)|<\epsilon$.
Since $\Delta_n$ tends to $0$, there exists $n_0 \in \mathbb{N}$ such that for every $n \geq n_0$, $\Delta_n <\delta$ holds.\\
Therefore, for $n\geq n_0$ we have that
\begin{align*}
&\Bigg|\frac{\Delta_n}{2}\sum_{i=1}^{A_n^t} \left(m(u_i^*,s_i^*)-m(u_i^*,u_i^*)\right)\Bigg| \leq \epsilon \frac{T}{2}.
\end{align*}
So \eqref{b3b} converges almost surely to $C_{jl}(t)$ so it converges in probability.

Finally, we proved that conditions of Theorem \ref{jac} are satisfied, hence the conclusion of theorem holds. We denote by $\pi_T$ the projection function $\pi_T : D([0,T], \mathbb{R}^d) \to \mathbb{R}^d$, defined by $\pi_T((X_t, t \in [0,T]))\coloneqq X_T$. Then by \cite[Theorem 12.5]{bill_cpm} it is a continuous function. Since we have the following relation
\begin{align*}
\frac{1}{\sqrt{\Delta_n}}\left(D\ell(\theta)-D\ell_n(\theta)\right)=V_n(\theta)+\pi_T\left(\left(\sum_{i=1}^{A_n^t}\chi_i^n\right)_{1\leq t\leq T}\right),
\end{align*}
the statement of Theorem \ref{main1} holds in the case of compact set $E$. 

In general, when $E \subseteq \mathbb{R}^k$ is open, there exists a set of open and bounded sets $(E_N)_{N \in \mathbb{N}}$ such that for every $N \in \mathbb{N}$, $Cl(E_N) \subset E_{N+1}$ and $E=\cup_{N=1}^{+\infty} E_N$. Without loss of generality, let $x_0 \in E_1$. There exists a sequence of $C^{\infty}(E)$-functions $\left(\phi_N\right)_{N\in \mathbb{N}}$ such that $\phi_N(x) \in [0,1]$ for all $x \in E$,  $\phi_N(x)=1$ on $Cl(E_N)$ and $\phi_N(x)=0$, on $Cl(E_{N+1})^C$\cite{ury}. Let us define $\mu^{(N)}(x,\theta)\coloneqq \phi_N(x)\mu(x,\theta)$. Let $\nu^{(N)}$ be continuous functions on $E$ such that $\nu^{(N)}(x) \coloneqq \nu(x)$ on $Cl(E_N)$ and $\nu^{(N)}(x)=K$ on $E \symbol{92} Cl(E_{N+1})$ where $K$ is a square root of some constant positive definite $k \times k$ matrix. Moreover, for $N \in \mathbb{N}$ we define $T_N \coloneqq \inf \{t\geq 0: X_t \in E_N^C\}$. Since $X$ is a continuous process, $(T_N)_{N \in \mathbb{N}}$ is an increasing sequence of stopping times and $T_N \nearrow +\infty$ a.s.\\
Let $N \in \mathbb{N}$ be fixed. Let process $X^{(N)} = (X_t^N; 0\leq t\leq T)$ be a diffusion process given as a strong solution of the following system of stochastic differential equations
\begin{align}\label{difN}
X_t^{(N)}=x_0+ \int_0^t \mu^{(N)}(X_s^{(N)},\theta_0)\,ds+\int_0^t \nu^{(N)}(X_s^{(N)})\, dW_s,\quad 0\leq t \leq T.
\end{align}
Under assumptions (A2) and (A4) functions $\mu^{(N)}(\cdot,\theta_0)$ and $\nu^{(N)}(\cdot)$ are bounded on $E$. They are also Lipschitz continuous. Using \cite[Corollary 5.1.2]{strvar} diffusion process \eqref{difN} exists and it is unique a.s. Let $\frac{1}{\sqrt{\Delta_n}} \left(D\ell^{(N)}(\theta)-D\ell_n^{(N)}(\theta)\right)$ be the term from the statement of the theorem for diffusion \eqref{difN}. Then, first part of the proof implies that $
\frac{1}{\sqrt{\Delta_n}} \left(D\ell^{(N)}(\theta)-D\ell_n^{(N)}(\theta)\right) \overset{st}{\Rightarrow}$ $ Y^{(N)}(\theta),\, n \to +\infty$ where $Y^{(N)}(\theta) \sim MN(0, \Sigma^{(N)}(\theta))$, and $\Sigma^{(N)}$ is a random matrix $\Sigma(\theta)$ that is applied on process $X^{(N)}$ and functions $\mu^{(N)}$ and $\nu^{(N)}$. Let us denote  $V_n(\theta)\coloneqq\frac{1}{\sqrt{\Delta_n}}\left(D\ell(\theta)-D\ell_n(\theta)\right)$ and $V_n^{(N)}(\theta)\coloneqq\frac{1}{\sqrt{\Delta_n}}\left(D\ell^{(N)}(\theta)-D\ell_n^{(N)}(\theta)\right)$. Let $f :  \mathbb{R}^d \to \mathbb{R}$ be a bounded and continuous function, and let $U$ be arbitrary bounded $\mathcal{F}_T$-measurable random variable. For almost all $\omega \in \Omega$ and $t \in \left[0, T^{(N)}\right]$, processes $X_t(\omega)$ and $X_t^{(N)}(\omega)$ are equal so we have:
\begin{align}
\Big|&\mathbb{E}\left[f(V_n(\theta))U\right]-\tilde{\mathbb{E}}\left[f(Y(\theta))U\right]  \Big|\leq \notag\\
&\Big|\mathbb{E}\left[f(V_n^{(N)}(\theta))U\mathbbm {1}_{\{T^{(N)}>T\}}\right]-\tilde{\mathbb{E}}\left[f(Y^{(N)}(\theta))U \mathbbm{1}_{\{T^{(N)}>T\}}\right]\Big| +\notag\\
&\Big|\mathbb{E}\left[f(V_n(\theta))U\mathbbm{1}_{\{T^{(N)} \leq T\}}\right]-\tilde{\mathbb{E}}\left[f(Y(\theta))U \mathbbm{1}_{\{T^{(N)} \leq T\}}\right]\Big|\label{part2}.
\end{align}
Using boundedness \eqref{part2} is bounded by $B\, \mathbb{P}\left(T^{(N)}\leq T\right)$ where $B$ is some positive constant. Using first part of the proof we have that
\begin{align*}
\overline{\lim\limits_{n}}\, \Big|&\mathbb{E}\left[f(V_n(\theta))U\right]-\tilde{\mathbb{E}}\left[f(Y(\theta))U\right]  \Big|\leq B\, \mathbb{P}\left(T^{(N)}\leq T\right).
\end{align*}
Letting $N \to +\infty$ we have
\begin{align*}
\overline{\lim\limits_{n}}\, \Big|&\mathbb{E}\left[f(V_n(\theta))U\right]-\tilde{\mathbb{E}}\left[f(Y(\theta))U\right]  \Big|=0
\end{align*}
that implies the statement of the theorem.
\end{proof}

\begin{proof} [Proof of Theorem \ref{main2}]
The idea of the proof is similar to the case when $k=1$ (see \cite[Theorem 5.4]{lubhuz}).
We denote by $Z_{n}(\theta)\equiv Z_{n}(\omega,\theta)\coloneqq \frac{1}{\sqrt{\Delta_n}}\left(D\ell(\theta)-D\ell_n(\theta)\right)$. Since we assumed that certain functions are smooth enough, using \cite[Lemma 4.1]{huzak2001}, we have that function  $(\omega,\theta) \mapsto Z_n(\omega, \theta)$ is $\mathcal{F}_T \otimes \mathcal{B}(\Theta)$, for all  $n \in \mathbb{N}$. Since MLE is a $\mathcal{F}_T$-measurable random vector \cite{huzak2001},  functions $\omega \mapsto Z_n(\hat{\theta})\equiv Z_n(\omega, \hat{\theta}(\omega))$ are $\mathcal{F}_T$-measurable.

Let $v \in \mathbb{R}^d$ be an arbitrary fixed vector and let $U$ be an arbitrary and almost surely bounded $\mathcal{F}_T$-measurable random variable. Let $B>0$ be a constant such that $|U|\leq B$ a.s. Using \cite[Lemma 4.3]{lubura}, it is sufficient to prove that
\begin{align*}
\lim_{n \to +\infty} \Big|\mathbb{E}\left[e^{i\left<v\,|\, Z_n(\hat{\theta})\right>}U\right]-\mathbb{E}\left[e^{-\frac{1}{2}\sum_{p,r=1}^d v_p v_r \Sigma_{pr}(\hat{\theta})}U\right]\Big|=0.
\end{align*}
For every $n \in \mathbb{N}$ we define functions: $
h_{n,1}(\omega, \theta)\coloneqq \cos \left(\left<v\,|\, Z_n(\omega,\theta)\right>\right)$,\\
$h_{n,2}(\omega, \theta)\coloneqq \sin \left(\left<v\,|\, Z_n(\omega,\theta)\right>\right)$ and
$h_{n,3}(\omega, \theta)\coloneqq e^{-\frac{1}{2}\sum_{p,r=1}^d v_pv_r \Sigma_{pr} (\omega, \theta)}.$
For $m=1, \dots, d$ we denote by $\partial_m h_{n,l}$ $m$-th partial derivative of $h_{n,l}$ with respect to $\theta_m$ for $l=1,2,3$.
Under assumption (A3), for every $\omega \in \Omega$, functions $\theta \mapsto h_{n,1} (\theta)\equiv h_{n,1}(\omega, \theta)$, $\theta \mapsto h_{n,2}(\theta) \equiv h_{n,2}(\omega, \theta)$ and $\theta \mapsto h_{n,3}\equiv h_{n,3}(\omega,\theta)$ are of class $C^1(\Theta)$. Using mean value theorem, we obtain \begin{align}\label{razlika_hl}
|h_{n,l}(\theta_1)-h_{n,l}(\theta_2)|\leq \left(\sum_{m=1}^d \sup_{\theta \in \Theta} |\partial_m h_{n,l}(\theta)|\right) \norm{\theta_2-\theta_1}_2.
\end{align}
It is easy to see that in order to bound $\partial_m h_{n,l}$, we have to bound $\partial_m Z_{n,j}$ (for $h_{n,1}$ and $h_{n,2}$) and $\partial_m \Sigma_{bc}$ (for $h_{n,3}$).
After a tedious calculation, we have
\allowdisplaybreaks
\begin{align}
\partial_m Z_{n,j}(\theta)
&=\frac{1}{\sqrt{\Delta_n}}\sum_{i=1}^n \int_{t_{i-1}}^{t_i} \bigl<S^{-1}(X_s)\partial_m\partial_j \mu(X_s,\theta)-\notag\\
&-S^{-1}(X_{i-1})\partial_m\partial_j \mu(X_{i-1},\theta)\,\big|\,\mu(X_s, \theta_0)\bigr>\,ds-\label{tip1}\\
&-\frac{1}{\sqrt{\Delta_n}}\sum_{i=1}^n \int_{t_{i-1}}^{t_i} \Bigl(\left<\partial_j\mu(X_s,\theta)\,\big|\,S^{-1}(X_s)\partial_m\mu(X_s,\theta)\right>-\notag\\
&-\left<\partial_j \mu(X_{i-1},\theta)\,\big|\,S^{-1}(X_{i-1})\partial_m\mu(X_{i-1},\theta)\right>\Bigr)\,ds-\label{tip2}\\
&-\frac{1}{\sqrt{\Delta_n}}\sum_{i=1}^n \int_{t_{i-1}}^{t_i} \Bigl(\left<\partial_m\partial_j\mu(X_s,\theta)\,\big|\,S^{-1}(X_s)\mu(X_s,\theta)\right>-\notag\\
&-\left<\partial_m\partial_j \mu(X_{i-1},\theta)\,\big|\,S^{-1}(X_{i-1})\mu(X_{i-1},\theta)\right> \Bigr)\,ds+\label{tip3}\\
&+\frac{1}{\sqrt{\Delta_n}}\sum_{i=1}^n \int_{t_{i-1}}^{t_i} \bigl<S^{-1}(X_s)\partial_m\partial_j \mu(X_s,\theta)-\notag\\
&-S^{-1}(X_{i-1})\partial_m\partial_j \mu(X_s,\theta)\,\big|\,\nu(X_s)\,dW_s\bigr>.\label{tip4}
\end{align}
Under assumption (A3), there exist $\partial_m \nabla g_j(x)$ for every $m, j=1,\dots,d$ so there exists
\begin{align}\label{dsig}
\partial_m\Sigma_{bc}(\theta)=\frac{1}{2}\int_0^T \sum_{p,r=1}^k S_{pr}(X_s) \partial_m &\bigl(\bigl<e_r^T \nabla g_b(X_s)\nu(X_s)\,| \notag\\
&\, e_p^T \nabla g_c(X_s)\nu(X_s)\bigr>\bigr)\,ds.
\end{align}
Assume for the moment that $E$ is a compact set, so all partial derivatives that appear in lines \eqref{tip1}-\eqref{dsig} are bounded functions on $E$.\\
Using \eqref{o_dt} in Theorem \ref{unif_b}, there exists a constant $K_{1,j,m}>0$ such that
\begin{align*}
\mathbb{E}\Bigl[\sup_{\theta \in \Theta}\big|\frac{1}{\sqrt{\Delta_n}}\sum_{i=1}^n &\int_{t_{i-1}}^{t_i} \bigl<S^{-1}(X_s)\partial_m\partial_j \mu(X_s,\theta)-\\
&-S^{-1}(X_{i-1})\partial_m\partial_j \mu(X_{i-1},\theta)\,\big|\,\mu(X_s, \theta_0)\bigr>\,ds\big|\Bigr]\leq K_{1,j,m}.
\end{align*}
Then, let us denote $\tilde{f}(x,\theta)\coloneqq\left<\partial_j \mu(x,\theta)\,\big|\,S^{-1}(x)\partial_m \mu(x, \theta)\right>$.\\ We construct a vector function $f=[\tilde{f},\cdots, \tilde{f}]^T$ and a vector function $a(x)\equiv[\frac{1}{k},\cdots, \frac{1}{k}]^T$. Using \eqref{o_dt} in Theorem \ref{unif_b}, there exists a constant $K_{2,j,m}>0$ such that
\begin{align*}
\mathbb{E}\Bigl[\sup_{\theta \in \Theta} \big|\frac{1}{\sqrt{\Delta_n}}\sum_{i=1}^n &\int_{t_{i-1}}^{t_i} \Bigl(\left<\partial_j\mu(X_s,\theta)\,\big|\,S^{-1}(X_s)\partial_m\mu(X_s,\theta)\right>-\\
&-\left<\partial_j \mu(X_{i-1},\theta)\,\big|\,S^{-1}(X_{i-1})\partial_m\mu(X_{i-1},\theta)\right>\Bigr)\,ds\big|\Bigr]\leq K_{2,j,m}.
\end{align*}
In the same manner, we conclude that there exists a constant $K_{3,j,m}>0$ such that
\begin{align*}
\mathbb{E}\Bigl[\sup_{\theta \in \Theta} \big| \frac{1}{\sqrt{\Delta_n}}\sum_{i=1}^n &\int_{t_{i-1}}^{t_i} \Bigl(\left<\partial_m\partial_j\mu(X_s,\theta)\,\big|\,S^{-1}(X_s)\mu(X_s,\theta)\right>-\\
&-\left<\partial_m\partial_j \mu(X_{i-1},\theta)\,\big|\,S^{-1}(X_{i-1})\mu(X_{i-1},\theta)\right> \Bigr)\,ds\big|\leq K_{3,j,m}.
\end{align*}
Using \eqref{o_dw} in Theorem \ref{unif_b}, there exists a constant $K_{4,j,m}>0$ such that
\begin{align*}
\mathbb{E}\Bigl[\sup_{\theta \in \Theta} \big|\frac{1}{\sqrt{\Delta_n}}\sum_{i=1}^n \int_{t_{i-1}}^{t_i} &\bigl<S^{-1}(X_s)\partial_m\partial_j \mu(X_s,\theta)-\\
&-S^{-1}(X_{i-1})\partial_m\partial_j \mu(X_s,\theta)\,\big|\,\nu(X_s)\,dW_s\bigr> \big| \Bigr] \leq K_{4,j,m}.
\end{align*}
Hence, for every $j, m=1,2, \dots, d$ we have that
\begin{align*}
\mathbb{E}\left[\sup_{\theta \in \Theta} \partial_m Z_{n,j}(\theta) \right]\leq K_{1,j,m}+K_{2,j,m}+K_{3,j,m} \eqqcolon K_{j,m}.
\end{align*}
Moreover, there exists a constant $L_{m,b,c}>0$ such that \begin{align*}
\mathbb{E}\left[\sup_{\theta \in \Theta} \partial_m\Sigma_{bc}(\theta)\right]\leq L_{m,b,c}.
\end{align*}
For $l=1,2,3$ let us denote $H_{n,l} \coloneqq \sum_{m=1}^d \sup_{\theta \in \Theta} |\partial_m h_{n,l}(\theta)|$ and $H_n \coloneqq H_{n,1}+H_{n,2}+H_{n,3}$. Let us also denote $F_n \coloneqq h_{n,1}+ih_{n,2}-h_{n,3}$. Using \eqref{razlika_hl}, we can easily see that for every $n \in \mathbb{N}$ we have
\begin{align*}
|F_{n}(\theta_1)-F_{n}(\theta_2)| &\leq |h_{n,1}(\theta_1)-h_{n,1}(\theta_2)|+|h_{n,2}(\theta_1)-h_{n,2}(\theta_2)|+\\
&+|h_{n,3}(\theta_1)-h_{n,3}(\theta_2)|\leq\\
&\leq (H_{n,1}+H_{n,2}+H_{n,3})\norm{\theta_1-\theta_2}_2=\\
&=H_n \norm{\theta_1-\theta_2}_2,
\end{align*}
where
\begin{align*}
\mathbb{E}\left[H_n\right]\leq 2\norm{v}_2 \sum_{m=1}^d \sum_{j=1}^d K_{j,m}+ \frac{1}{2}\norm{v}_2^2 \sum_{m=1}^d \sum_{b,c=1}^d L_{m,b,c}\coloneqq K
\end{align*}
and constant $K$ does not depend on $n$.\\
Using previously introduced notation, it is sufficient to prove that \begin{align*}
\lim_{n \to +\infty} \Big|\mathbb{E}\left[F_n(\hat{\theta})U\right] \Big|=0.
\end{align*}
Let $\varepsilon>0$ be arbitrary and fixed. We define $\delta\coloneqq \frac{\varepsilon}{2KB}$. Since $Cl(\Theta)$ is compact set, there exists its finite cover, i.e. there exist finitely many balls $K(\theta_l, \delta)$, $l=1, \dots, N$, such that $\theta_l \in \Theta$ and $\Theta \subseteq Cl(\Theta) \subseteq \cup_{l=1}^{N} K(\theta_l, \delta)$. We define a finite partition  $\{K_1, \dots, K_N\}$ of set $\Theta$ in the following way:
\begin{align*}
K_1 &\coloneqq K(\theta_1, \delta) \cap \Theta,\\
K_2 &\coloneqq K(\theta_2, \delta) \cap \Theta \cap K_1^C,\\
\vdots\\
K_N &\coloneqq K(\theta_N, \delta) \cap \Theta \cap K_1^C \dots\cap K_{N-1}^C.
\end{align*}
\begin{align}
\Big|\mathbb{E} \left[F_n(\hat{\theta})U\right]\Big|
&\leq\Big|\sum_{l=1}^N \mathbb{E}\left[(F_n(\hat{\theta})-F_n(\theta_l))U \mathbbm{1}_{\{\hat{\theta} \in K_l\}}\right]\Big|+\label{partition}\\
&+\Big|\sum_{l=1}^N \mathbb{E}\left[F_n(\theta_l)U\mathbbm{1}_{\{\hat{\theta} \in K_l\}}\right]\Big|\notag
\end{align}
On the event $\{\hat{\theta} \in K_l\}$ holds $\norm{\hat{\theta}-\theta_l}_2< \delta$ so for the first summand in  \eqref{partition} we have that
\begin{align*}
\Big|\sum_{l=1}^N \mathbb{E}\left[(F_n(\hat{\theta})-F_n(\theta_l))U \mathbbm{1}_{\{\hat{\theta} \in K_l\}}\right]\Big|&\leq \sum_{l=1}^N \mathbb{E}\left[|F_n(\hat{\theta})-F_n(\theta_l)||U|\mathbbm{1}_{\{\hat{\theta}\in K_l\}}\right]\leq\\
&\leq B \sum_{l=1}^N \mathbb{E}\left[H_n \norm{\hat{\theta}-\theta_l}_2\mathbbm{1}_{\{\hat{\theta}\in K_l\}}\right]<\\
&< B\delta \sum_{l=1}^N \mathbb{E}\left[H_n \mathbbm{1}_{\{\hat{\theta} \in K_l\}}\right]=BK\delta=\frac{\varepsilon}{2}.
\end{align*}
Since $\hat{\theta}$ is a $\mathcal{F}_T$-measurable random vector,  $U \mathbbm{1}_{\{\hat{\theta}\in K_l\}}$ is bounded a $\mathcal{F}_T$-measurable random variable for every $l=1,\dots, N$. For the second summand in \eqref{partition} Theorem \ref{main1} and \cite[Lemma 4.3]{lubura} imply
\begin{align*}
\lim_{n \to +\infty} \mathbb{E}\left[F_n(\theta_l)U \mathbbm{1}_{\{\hat{\theta}\in K_l\}}\right]=0.
\end{align*}
Moreover,
\begin{align*}
\lim_{n \to +\infty} \Big|\sum_{l=1}^N \mathbb{E}\left[F_n(\theta_l)U \mathbbm{1}_{\{\hat{\theta} \in K_l\}}\right] \Big|=0
\end{align*}
holds.\\
Finally, now we may choose $n_0=n_0(\varepsilon) \in \mathbb{N}$ such that for all $n \geq n_0$
\begin{align*}
\Big|\sum_{l=1}^N \mathbb{E}\left[F_n(\theta_l)U\mathbbm{1}_{\{\hat{\theta}\in K_l\}}\right]\Big|<\frac{\varepsilon}{2}.
\end{align*}
For all $n \geq n_0$ it follows that
\begin{align*}
\Big|\mathbb{E} \left[F_n(\hat{\theta})U\right]\Big| \leq \frac{\varepsilon}{2}+\frac{\varepsilon}{2}=\varepsilon.
\end{align*}
The statement of the theorem holds when $E$ is compact. Using the same construction with stopping times $\left(T_{N}\right)_{N \in \mathbb{N}}$ as in Theorem \ref{main1}, we achieve the general statement.
\end{proof}
\begin{proof} [Proof of Lemma \ref{konstr}]
Without loss of generality, we prove the assertion for $r=0$. The difference between $\ell_n$ and $\ell$ equals
\begin{align}
\ell_n(\theta)-\ell(\theta)&=\sum_{i=1}^n \int_{t_{i-1}}^{t_i} \bigl<S^{-1}(X_{i-1})\mu(X_{i-1}, \theta)-\notag\\
&-S^{-1}(X_s)\mu(X_s, \theta)\,\big|\,dX_s\bigr>-\label{diff1}\\
&-\frac{1}{2}\sum_{i=1}^n \int_{t_{i-1}}^{t_i} \Bigl(\left<\mu(X_{i-1},\theta)\,|\,S^{-1}(X_{i-1})\mu(X_{i-1},\theta)\right>-\notag\\
&-\left<\mu(X_s,\theta)\,|\, S^{-1}(X_s)\mu(X_s,\theta)\right>\Bigr)\,ds.\label{diff2}
\end{align}
Let us assume for the moment that $E$ is compact set. Relations \eqref{o_dt} and \eqref{o_dw} in Theorem \ref{unif_b} and \eqref{sde} imply that for \eqref{diff1} there exists a constant $C>0$ such that
\begin{align*}
\norm{\sup_{\theta \in \Theta}\Big|\sum_{i=1}^n\int_{t_i}^{t_{i+1}}\left<S^{-1}(X_{i-1})\mu(X_{i-1},\theta)-S^{-1}(X_s)\mu(X_s,\theta)\,\big|\,dX_s\right>\Big|}_{L^2}
\end{align*}
is bounded above by $C \sqrt{\Delta_n}$.
Constructing a vector function in the same way as in the proof of Theorem \ref{main2}, we may use Theorem \ref{unif_b} again for \eqref{diff2}. Finally, there exists a constant $C_X$ (which depends on process $X$) such that
\begin{align*}
\norm{\sup_{\theta \in \Theta}|\ell_n(\theta)-\ell(\theta)|}_{L^2} \leq C_X \sqrt{\Delta_n}.
\end{align*}
In the case when $E$ is an open set or functions that appear in calculations are not bounded, the same construction from proof of Theorem \ref{main1} has been imposed. For an open set $E$ there exists a sequence of open and bounded sets $(E_N)_{N \in \mathbb{N}}$
such that for all $N \in \mathbb{N}$,  $Cl(E_N) \subset E_{N+1}$,  $E=\cup_{N=1}^{+\infty} E_N$, and $x_0 \in E_1$. Also, there exists a sequence of $C^{\infty}$-functions $\left(\phi_N\right)_{N\in \mathbb{N}}$ such that $\phi_N(x) \in [0,1]$ for all $x \in E$, and $\phi_N(x)=1$ on $Cl(E_N)$ and $\phi_N(x)=0$ on $Cl(E_{N+1})^C$.\\
As before, we define functions $\mu^{(N)}(x,\theta)\coloneqq \phi_N(x)\mu(x,\theta)$ and $\nu^{(N)}(x) \coloneqq \nu(x)$ on $Cl(E_N)$ and $\nu^{(N)}(x)=K$ on $E \symbol{92} Cl(E_{N+1})$ where $K$ is a square root of some constant positive definite $k \times k$ matrix. Also, we define a sequence of stopping times $\left(T_N\right)_{N \in \mathbb{N}}$ and observe the diffusion process $X^{(N)}$ that is connected to $T_N$.\\
By the same reasoning as in the proof of Theorem \ref{main1}, the diffusion process \eqref{difN} exists and it is unique a.s.
Let us denote the diffusion matrix of $X^{(N)}$ by $S^{(N)}(x)\coloneqq \nu^{(N)}(x)\nu^{(N)}(x)^T$. Under assumptions (A2-4) function $(x,\theta)\mapsto (S^{(N)})^{-1}(x)\mu^{(N)}(x, \theta)$ satisfies the assumptions of Theorem \ref{unif_b}. Using first part of this proof, it follows that
\begin{align*}
\norm{\sup_{\theta \in \Theta}|\ell_n^{(N)}(\theta)-\ell^{(N)}(\theta)|}_{L^2}\leq C_{X^{(N)}} \sqrt{\Delta_n},
\end{align*}
where constant $C_{X^{(N)}}$ depends on process $\left(X_t^{(N)}\right)_{0\leq t\leq T}$.\\
Let $P_{x_0}^{(N)}$ be a distribution of solution \eqref{difN}, and $P_{x_0}$ a distribution of solution \eqref{sde}. It follows that $P_{x_0}^{(N)}(\cdot)=P_{x_0}(\cdot)$ on $\sigma$-algebra $\sigma\left(X_{s \wedge T_N}: s \geq 0\right)$ \cite[Corollary 10.1.2]{strvar}.
Hence,
\allowdisplaybreaks
\begin{align*}
&\mathbb{P} \left(\sup_{\theta \in \Theta} |\ell_n(\theta)-\ell(\theta)|>A\sqrt{\Delta_n}\right)\leq\\
&\leq \mathbb{P}\left(\{\sup_{\theta \in \Theta} |\ell_n(\theta)-\ell(\theta)|>A\sqrt{\Delta_n}, \, T_N \leq T\}\right)+\\
&+\mathbb{P} \left(\sup_{\theta \in \Theta} |\ell_n^{(N)}(\theta)-\ell^{(N)}(\theta)|>A\sqrt{\Delta_n}\right)\leq\\
&\leq \mathbb{P}\left(T_N\leq T\right) + \frac{1}{A\sqrt{\Delta_n}}\mathbb{E}\left[|\sup_{\theta \in \Theta} |\ell_n^{(N)}(\theta)-\ell^{(N)}(\theta)|\right]\leq \\
&\leq \mathbb{P}\left(T_N\leq T\right) + \frac{1}{A\sqrt{\Delta_n}} \norm{\sup_{\theta \in \Theta} |\ell_n^{(N)}(\theta)-\ell^{(N)}(\theta)|}_{L^2}\leq\\
&=\mathbb{P}\left(T_N\leq T\right)+\frac{1}{A}C_{X^{(N)}}.
\end{align*}
Let first $n \to +\infty$, then $A \to +\infty $ and finally $N \to +\infty$. We conclude that
\begin{align*}
\lim_{A \to +\infty}\lim_{n \to +\infty}\mathbb{P}\left(\sup_{\theta \in \Theta} \frac{|\ell_n(\theta)-\ell(\theta)|}{\sqrt{\Delta_n}}>A\right)=0.
\end{align*}
\end{proof}
\begin{lemma}
Let $\Theta$ be a convex and relatively compact set. Assume (A1-6). Then
\begin{align*}
\int_0^1 D^2\ell_n(\hat{\theta}+(\overline{\theta}_n-\hat{\theta})z)\,dz \cdot \frac{1}{\sqrt{\Delta_n}} \left(\overline{\theta}_n-\hat{\theta}\right) \overset{st}{\Rightarrow} Y(\hat{\theta}),
\end{align*}
where $Y(\hat{\theta}) \sim MN(0,\Sigma(\hat{\theta}))$.
\end{lemma}
\begin{lemma}
Let $\Theta$ be a convex and relatively compact set. Assume (A1-6). Then
\begin{align*}
\sup_{z \in [0,1]} \big| D^2\ell(\hat{\theta})-D^2\ell(\hat{\theta}+z(\overline{\theta}_n-\hat{\theta}))\big|\overset{\mathbb{P}}{\to} 0.
\end{align*}
\end{lemma}
\begin{lemma}\label{last_l}
Let $\Theta$ be a convex and relatively compact set. Assume (A1-6). Then
\begin{align*}
D^2\ell(\hat{\theta})\frac{1}{\sqrt{\Delta_n}}\left(\overline{\theta}_n-\hat{\theta}\right) \overset{st}{\Rightarrow} MN(0, \Sigma(\hat{\theta})).
\end{align*}
\end{lemma}
The last three lemmas are the same as in the case of a one-dimensional diffusion \cite{lubhuz} and so are their proofs.
\begin{proof}[Proof of Theorem \ref{the_core}]
Since $D^2\ell(\hat{\theta})$ is a negatively definite matrix (A6), it is a regular matrix, so its inverse exists. $\left(D^2\ell(\hat{\theta})\right)^{-1}$ is also a symmetric matrix. Finally, the desired result follows from Lemma \ref{last_l}.
\end{proof}

\section{Example and simulations}
The stochastic model that we used for simulations is the Heston model analyzed in \cite{hes}. The model is used for the  analysis of financial data and is given by
\begin{align}\label{ex}
dY_t&=(a-bY_t)\,dt+\sigma_1 \sqrt{Y_t}\,dW_t^1\\\notag
dX_t&=(\alpha-\beta Y_t)\,dt+\sigma_2 \sqrt{Y_t}\left(\rho\,dW_t^1+\sqrt{1-\rho^2}\,dW_t^2\right),
\end{align}
where $a>0$, $b, \alpha, \beta \in \mathbb{R}$, $\sigma_1>0$, $\sigma_2>0$, $\rho \in \left<-1,1\right>$ and $\left(W_t^1, W_t^2\right)_{t \geq 0}$ is a two-dimensional standard Wiener process.
Although this model does not satisfy the assumption of uniform ellipticity, it is shown in \cite[Proposition 2.1]{hes} that there exists a pathwise unique strong solution of \eqref{ex} for $t \geq 0$, and that the log-likelihood function can be written using results from \cite[Section 7, p. 296]{lipsh}, as was done in Section 3 of \cite{hes}.
For simulation purposes, a version of Theorem \ref{the_core} is used in which $\Sigma(\hat{\theta})$ and $D^2 \ell(\hat{\theta})$ are replaced by $\Sigma_n(\overline{\theta}_n)$ and $D^2 \ell_n(\overline{\theta}_n)$, respectively. Matrix $\Sigma_n(\overline{\theta}_n)$ whose elements are given by
\begin{align*}
    \Sigma_n({\theta})_{jl}=\frac{\Delta_n}{2}\sum_{i=1}^n \sum_{p,r=1}^k S_{pr}(X_{i-1})  &\bigl<e_r^T \nabla g_j(X_{i-1})\nu(X_{i-1}) \,\big|\\
    &e_p^T \nabla g_l(X_{i-1})\nu(X_{i-1})\bigr>
\end{align*}
can be understood as the discretized version of matrix $\Sigma(\theta)$. Since the drift function in \eqref{ex} is linear in $\theta$, matrices $\Sigma(\theta)$ and $\Sigma_n(\theta)$ do not depend on $\theta$. The proof of such a version of Theorem \ref{the_core} is similar to the proofs of Lemma 5.9 and 5.11 in \cite{lubhuz}.\\
MLE $\hat{\theta}$ of the process in \eqref{ex} is explicitly given by
\begin{align}\label{mle_ex}
    \begin{bmatrix}
    \hat{a} \\ \hat{b} \\ \hat{\alpha} \\ \hat{\beta}
    \end{bmatrix}
    =\frac{1}{\int_0^T Y_s \,ds \int_0^T \frac{ds}{Y_s}-T^2}\begin{bmatrix}
    \int_0^T Y_s\,ds \int_0^T \frac{dY_s}{Y_s}-T(Y_T-y_0)\\
    T\int_0^T \frac{dY_s}{Y_s}-(Y_T-y_0)\int_0^T \frac{ds}{Y_s}\\
    \int_0^T Y_s\,ds \int_0^T \frac{dX_s}{Y_s}-T(X_T-x_0)\\
    T\int_0^T \frac{dX_s}{Y_s}-(X_T-x_0)\int_0^T \frac{ds}{Y_s}
    \end{bmatrix}
\end{align}
if $\int_0^T Y_s\,ds \int_0^T ds/Y_s > T^2$. This condition is satisfied if the parameters of the model are such that $a \in \left[\sigma_1^2/2, \, +\infty\right>$, $b \in \mathbb{R}$, $\sigma_1 >0$ and $y_0>0$. Moreover, MLE in \eqref{mle_ex} is unique when $\alpha, \beta$ and $x_0$ are real numbers, $\sigma_2>0$ and $\rho \in \left<-1,1\right>$\cite{hes}.\\
We define the vector function $\mu(x,y,a,b, \alpha, \beta)$ and the matrix function $\nu(x,y)$ as
\begin{align*}
\mu(x,y,a,b, \alpha, \beta)=\begin{bmatrix}
a-by\\
\alpha-\beta y
\end{bmatrix}, \quad \quad
\nu(x,y)=\begin{bmatrix}
\sigma_1\sqrt{y} & 0\\
\sigma_2 \rho\sqrt{y}& \sigma_2 \sqrt{(1-\rho^2)y}
\end{bmatrix}.
\end{align*}
As explained in Section 3, we simulate $M$ realizations of discrete random sample $\left(Y_{t_i},X_{t_i}\right)_i$, for $i=1,2,\dots,N$ and $N=2^l$,  with parameters $a=2$, $b=-0.8$, $\alpha=0.02$, $\beta=2$, $\sigma_1=0.7$, $\sigma_2=0.6$, $\rho=-0.8$, $x_0=\ln{100}$ and $y_0=0.5$. The time interval is $\left[0, T\right]$ for $T=1$ and the subdivision of points is equidistant, $t_i=\frac{T}{N}i$ for $i=0, \dots, N$.\\
Because of the linearity of the function $\mu$ in the parameters and the second summand in \eqref{lln} the discrete log-likelihood is a quadratic function in $a, b, \alpha$ and $\beta$, thus it has a maximum. Using the abbreviations $X_{t_i}=X_i$ and $Y_{t_i}=Y_i$, the AMLE for the vector of drift parameters $\overline{\theta}_n$ is given by
\begin{align}\label{amle_ex}
\begin{bmatrix}
    \overline{a}_n \\ \overline{b}_n \\ \overline{\alpha}_n \\ \overline{\beta}_n
    \end{bmatrix}
    =F\cdot\begin{bmatrix}
    \Delta_n \sum_{i=1}^n Y_{i-1} \sum_{i=1}^{n} \frac{Y_i-Y_{i-1}}{Y_{i-1}}-T(Y_n-y_0)\\
    T \sum_{i=1}^n \frac{Y_i-Y_{i-1}}{Y_{i-1}}-\Delta_n(Y_n-y_0)\sum_{i=1}^n \frac{1}{Y_{i-1}}\\
    \Delta_n \sum_{i=1}^n Y_{i-1} \sum_{i=1}^n \frac{X_{i}-X_{i-1}}{Y_{i-1}}-T(X_n-x_0)\\
    T\sum_{i=1}^n \frac{X_i-X_{i-1}}{Y_{i-1}}-\Delta_n(X_n-x_0)\sum_{i=1}^n \frac{1}{Y_{i-1}}
    \end{bmatrix}
\end{align}
where $F=\left(\Delta_n^2 \sum_{i=1}^n Y_{i-1} \sum_{i=1}^n \frac{1}{Y_{i-1}}-T^2\right)^{-1}$. After long and tedious calculation, we obtain the matrices $\Sigma_n(\overline{\theta}_n)$ and $D^2 l_n(\overline{\theta}_n)$. The formulas are
\begin{align*}
\Sigma_n(\overline{\theta}_n)&=\frac{\Delta_n}{2}\sum_{i=1}^n\frac{1}{Y_{i-1}^2}\cdot
    \begin{bmatrix}
    \frac{1}{1-\rho^2} & 0 & \frac{-\sigma_1 \rho}{\sigma_2(1-\rho^2)}& 0\\
    0 & 0 & 0 & 0\\
    \frac{-\sigma_1 \rho}{\sigma_2(1-\rho^2)} & 0 & \frac{\sigma_1^2}{\sigma_2^2 (1-\rho^2)} & 0\\
    0 & 0 & 0 & 0
    \end{bmatrix},\\
D^2l_n(\overline{\theta}_n)&=G \cdot \begin{bmatrix}
    \sigma_2^2 & -\sigma_1 \sigma_2 \rho \\
    -\sigma_1 \sigma_2 \rho & \sigma_1^2
    \end{bmatrix} \otimes \begin{bmatrix}
    -\Delta_n \sum_{i=1}^n \frac{1}{Y_{i-1}} & T\\
    T & -\Delta_n \sum_{i=1}^n Y_{i-1}
    \end{bmatrix},
\end{align*}
where $G=1/(\sigma_1^2 \sigma_2^2 (1-\rho^2))$ and $\otimes$ denotes the Kronecker product of two matrices \cite{gentle}. 

Since MLE can not be calculated using \eqref{mle_ex}, we estimate it with the formulas in \eqref{amle_ex}. Then we compute AMLE also with \eqref{amle_ex}, but with fewer points than for MLE. More precisely, we take a subsample of length $n=2^k, k<l$, and $\Delta_n=\frac{T}{n}$. Then we determine the percentage of values $\norm{\frac{1}{\sqrt{\Delta_n}} \sqrt{\Sigma_n(\overline{\theta}_n)}^{+}D^2l_n(\overline{\theta}_n)\left(\overline{\theta}_n-\hat{\theta}\right)}_2^2$ that are in the interval $\left[0,\, \chi_{1-p}^2(r)\right]$ where $\chi_{1-p}^2(r)$ is $(1-p)$-quantile of $\chi^2$-distribution with $r$ degrees of freedom. The degrees of freedom $r$ correspond to the rank of the covariance matrix of the observed expression in the norm. The matrix $\Sigma_n(\overline{\theta}_n)$ is a symmetric and singular matrix (and $\Sigma(\hat{\theta})$ as well). From the fact that $\rho \in \left<-1,\,1\right>$, we conclude that its submatrix of nonzero elements is a strictly positive definite matrix of rank 2. Since symmetric matrices are orthogonally diagonalizable, the eigenvalue decomposition of the submatrix is the same as its singular value decomposition (SVD) \cite{gentle}. The SVD of the matrix $\Sigma_n(\overline{\theta}_n)$ is then obtained by placing zeros on the diagonal matrix of SVD of the submatrix and adding vectors to the set of its eigenvectors to complete the basis of $\mathbb{R}^4$. Let us denote the resulting SVD by $UDU^T$. It is now easy to see that the generalized inverse of the square root of $\Sigma_n(\overline{\theta}_n)$ is given by $U\tilde{D}U^T$ where $\tilde{D}$ is such a diagonal matrix that for all $i \leq 2$, $\tilde{d}_{ii}=1/\sqrt{d_{ii}}$ and for $i>2$, $\tilde{d}_{ii}=0$. For this reason, after a short calculation, we conclude that in our case the covariance matrix of  $\frac{1}{\sqrt{\Delta_n}} \sqrt{\Sigma_n(\overline{\theta}_n)}^{+}D^2l_n(\overline{\theta}_n)\left(\overline{\theta}_n-\hat{\theta}\right)$ is an identical matrix of rank $r=2$.\\

\begin{table}[ht]
\caption{$M=1000$, $p=0.025$, $l=12$}
\label{table1}
\centering
\begin{tabular}{|c||  c c c c c c |}
 \hline
$k$  & 3 & 4 & 5 & 6 & 7 & 8  \\ [0.5ex]
 \hline
 \%  & 0.675 & 0.812 & 0.867 & 0.942 & 0.952 & 0.972\\
 \hline
\end{tabular}
\end{table}

\begin{table}[ht]
\caption{$M=1000$, $p=0.05$, $l=14$}
\label{table2}
\centering
\begin{tabular}{|c||  c c c c c c c c |}
 \hline
$k$  & 3 & 4 & 5 & 6 & 7 & 8 & 9 & 10 \\ [0.5ex]
 \hline
 \% & 0.589 & 0.740 & 0.826 & 0.884 & 0.911 & 0.935 &0.944 &0.95\\
 \hline
\end{tabular}
\end{table}

\begin{table}[ht]
\caption{$M=1000$, $p=0.05$, $l=16$}
\label{table3}
\centering
\begin{tabular}{|c||  c c c c c c c c c |}
 \hline
$k$  & 3 & 4 & 5 & 6 & 7 & 8 & 9 & 10 & 11   \\ [0.5ex]
 \hline
 \% & 0.596 & 0.726 & 0.831 & 0.904 & 0.921 & 0.931 &0.933 &0.939 & 0.955 \\
 \hline
\end{tabular}
\end{table}

Tables \ref{table1}, \ref{table2} and \ref{table3} show that an increase in $k$ causes an increase in the percentage so that the value, for a given $p$, approaches $1-p$.

\section{Appendix}
\begin{proof} [Proof of Lemma \ref{lemma1}]
Let $c_1, \dots, c_{7}$ be positive constants occurring in the proof. For function $g_j(x) \equiv g_j(x, \theta)$ ($(x,\theta )\in E\times\Theta$) as in \eqref{gj}, under (A2) and (A4) we use Lemma \ref{mito} in \eqref{a1} and \eqref{a3}.
\allowdisplaybreaks
\begin{align}
\frac{1}{\sqrt{\Delta_n}} \sum_{i=1}^n &\int_{t_{i-1}}^{t_i} \left< g_j(X_s)-g_j(X_{i-1})\,|\, \mu(X_s, \theta_0) \right> \,ds= \notag\\
&=\frac{1}{\sqrt{\Delta_n}}\sum_{i=1}^n \int_{t_{i-1}}^{t_i} \Bigl< \int_{t_{i-1}}^s \Bigl(\nabla g_j (X_u) \mu(X_u, \theta_0)+\notag\\
&+\frac{1}{2} \nabla_2 g_j(X_u)\Bigr) \, du\,\bigg| \, \mu(X_s, \theta_0) \Bigr> \,ds+ \label{a11}\\
&+\frac{1}{\sqrt{\Delta_n}} \sum_{i=1}^n \int_{t_{i-1}}^{t_i} \left<\int_{t_{i-1}}^s \nabla g_j(X_u) \nu(X_u) \, dW_u\,\bigg| \, \mu(X_s, \theta_0) \right>\, ds. \label{a12}
\end{align}
\begin{align}
\frac{1}{\sqrt{\Delta_n}} \sum_{i=1}^n &\int_{t_{i-1}}^{t_i} \left<g_j(X_s)-g_j(X_{i-1})\,| \, \nu(X_s) \, dW_s\right>=\notag\\
&=\frac{1}{\sqrt{\Delta_n}} \sum_{i=1}^n \int_{t_{i-1}}^{t_i} \Bigl< \int_{t_{i-1}}^s \Bigl(\nabla g_j (X_u) \mu(X_u, \theta_0)+\notag\\
&+\frac{1}{2} \nabla_2 g_j(X_u)\Bigr) \, du\,\bigg| \, \nu(X_s) \,dW_s\Bigr>+\label{a31}\\
&+\frac{1}{\sqrt{\Delta_n}} \sum_{i=1}^n \int_{t_{i-1}}^{t_i} \left< \int_{t_{i-1}}^s \nabla g_j(X_u)\nu(X_u) \, dW_u\,\bigg| \, \nu(X_s) \, dW_s\right>\notag
\end{align}
Also, under (A3) and (A4) we use Lemma \ref{vito} in \eqref{a2} for function:
\begin{align}\label{fj}
f_j(x)\equiv f_j (x,\theta ) \coloneqq \left<S^{-1}(x)\mu(x,\theta)\,|\, \partial_j \mu(x, \theta)\right>,\;\;
(x,\theta )\in E\times\Theta.
\end{align}
\begin{align}
&\frac{1}{\sqrt{\Delta_n}} \sum_{i=1}^n \int_{t_{i-1}}^{t_i} \left(f_j(X_{i-1})-f_j(X_s)\right)\,ds=\notag \\
&=-\frac{1}{\sqrt{\Delta_n}} \sum_{i=1}^n \int_{t_{i-1}}^{t_i} \left(f_j(X_s)-f_j(X_{i-1})\right)\,ds=\notag\\
&=-\frac{1}{\sqrt{\Delta_n}} \sum_{i=1}^n \int_{t_{i-1}}^{t_i} \int_{t_{i-1}}^s \bigl(\left<\nabla f_j(X_u)\,|\, \mu(X_u, \theta_0)\right>+\notag\\
&+\frac{1}{2}\Tr{\left(S(X_u)\nabla\left(\nabla f_j(X_u))\right)\right)}\bigr)\,du\,ds- \label{a21}\\
&-\frac{1}{\sqrt{\Delta_n}} \sum_{i=1}^n \int_{t_{i-1}}^{t_i} \int_{t_{i-1}}^s \left<\nabla f_j(X_u)\,| \, \nu(X_u) \, dW_u \right>\, ds \label{a22}
\end{align}
For each $j=1,2,\dots, d$, we define a component of vector $V_n$ as $V_n^j(\theta)=$\eqref{a11}+\\+\eqref{a12}+\eqref{a31}+\eqref{a21}+\eqref{a22}. We will prove that for all $j=1,2, \dots, d$ $V_n^j$ converges in probability to zero, when $n \rightarrow +\infty$. Obviously, then vector $V_n(\theta)$ converges in probability to $\mathbf{0}_d$, when $n \rightarrow +\infty$.

For simplicity, we propose some new notation: $G_j(x, \theta)\coloneqq\nabla g_j(x)\mu(x,\theta)+\frac{1}{2}\nabla_2 g_j(x).$
Using Cauchy-Schwartz inequality for vectors and boundedness of functions we conclude that for \eqref{a11} there exist constant $c_1$ such that
\begin{align*}
\Bigg| \frac{1}{\sqrt{\Delta_n}} \sum_{i=1}^{n} \int_{t_{i-1}}^{t_i} \left< \int_{t_{i-1}}^s G_j(X_u, \theta_0) \, du\,\bigg| \, \mu(X_s, \theta_0) \right> \,ds \Bigg| \leq c_1 \sqrt{\Delta_n}T.
\end{align*}
In the same manner, we conclude that for \eqref{a21} there exists constant $c_2$ such that
\begin{align*}
\Bigg| -\frac{1}{\sqrt{\Delta_n}} \sum_{i=1}^n \int_{t_{i-1}}^{t_i} \int_{t_{i-1}}^s &\Bigl(\left<\nabla f_j(X_u)\,|\, \mu(X_u, \theta_0)\right>+\\
&+\frac{1}{2}\Tr{\left(S(X_u)\nabla\left(\nabla f_j(X_u))\right)\right)}\Bigr)\,du\,ds \Bigg|\leq  c_2 \sqrt{\Delta_n}T.
\end{align*}
In the sequel, we analyze \eqref{a12} using Lemma \ref{mito} for function $x \mapsto \mu(x, \theta_0)$.
\allowdisplaybreaks
\begin{align}
\frac{1}{\sqrt{\Delta_n}} \sum_{i=1}^n &\int_{t_{i-1}}^{t_i} \left<\int_{t_{i-1}}^s \nabla g_j(X_u) \nu(X_u) \, dW_u\,\bigg| \, \mu(X_s, \theta_0) \right>\, ds =\notag\\
&=\frac{1}{\sqrt{\Delta_n}} \sum_{i=1}^n \int_{t_{i-1}}^{t_i} \left< \int_{t_{i-1}}^s \nabla g_j(X_u) \nu(X_u) \, dW_u\,\bigg| \, \mu(X_{i-1}, \theta_0)\right> \,ds+ \label{a12a}\\
&+\frac{1}{\sqrt{\Delta_n}} \sum_{i=1}^n \int_{t_{i-1}}^{t_i} \Bigl<\int_{t_{i-1}}^s \nabla g_j(X_u) \nu(X_u) \, dW_u\,\bigg|\notag\\ 
&\int_{t_{i-1}}^{s} \nabla \mu(X_u, \theta_0)\mu(X_u, \theta_0)\, du\Bigr> \,ds+\label{a12b}\\
&+ \frac{1}{\sqrt{\Delta_n}} \sum_{i=1}^n \int_{t_{i-1}}^{t_i} \Bigl<\int_{t_{i-1}}^s \nabla g_j(X_u) \nu(X_u) \, dW_u\,\bigg|\notag\\ &\frac{1}{2} \int_{t_{i-1}}^s
\nabla_2 \mu(X_u, \theta_0)\,du\Bigr>\,ds +\label{a12c}\\
&+\frac{1}{\sqrt{\Delta_n}} \sum_{i=1}^n \int_{t_{i-1}}^{t_i} \Bigl<\int_{t_{i-1}}^s \nabla g_j(X_u) \nu(X_u) \, dW_u\,\bigg| \notag\\ 
&\int_{t_{i-1}}^{s} \nabla \mu(X_u, \theta_0) \nu(X_u) \, dW_u \Bigr>\,ds\label{a12d}
\end{align}

For \eqref{a12a} and \eqref{a12b} + \eqref{a12c} it is sufficient to prove that it converges to zero in $L^2$. Then it converges in probability to zero, too. The $L^2$ convergence of \eqref{a12a} is proved in the sequel.
\begin{align*}
&\mathbb{E}\left[\left(\frac{1}{\sqrt{\Delta_n}}\sum_{i=1}^n \int_{t_{i-1}}^{t_i} \left<\int_{t_{i-1}}^s \nabla g_j(X_u)\nu(X_u) \, dW_u\,\bigg| \, \mu(X_{i-1}, \theta_0)\right>\,ds\right)^2\right]=\\
&=\frac{1}{\Delta_n}\sum_{i=1}^n \mathbb{E}\left(\int_{t_{i-1}}^{t_i} \left<\int_{t_{i-1}}^s \nabla g_j(X_u)\nu(X_u) \, dW_u\,\bigg| \, \mu(X_{i-1}, \theta_0)\right>\,ds\right)^2+\\
&+\frac{2}{\sqrt{\Delta_n}} \sum_{1 \leq i < l \leq n} \mathbb{E}\Biggl[\left(\int_{t_{i-1}}^{t_i} \left<\int_{t_{i-1}}^s \nabla g_j(X_u)\nu(X_u) \, dW_u\,\bigg| \, \mu(X_{i-1}, \theta_0)\right>\,ds\right)\cdot \\
&\cdot \left(\int_{t_{l-1}}^{t_l} \left<\int_{t_{l-1}}^s \nabla g_j(X_u)\nu(X_u) \, dW_u\,\bigg| \, \mu(X_{l-1}, \theta_0)\right>\,ds\right)\Biggr]\\
&=\frac{1}{\Delta_n}\sum_{i=1}^n \mathbb{E}\left(\int_{t_{i-1}}^{t_i} \left<\int_{t_{i-1}}^s \nabla g_j(X_u)\nu(X_u) \, dW_u\,\bigg| \, \mu(X_{i-1}, \theta_0)\right>\,ds\right)^2\leq \\
&\leq \frac{1}{\Delta_n}\sum_{i=1}^n \mathbb{E}\left(\int_{t_{i-1}}^{t_i} \left(\left<\int_{t_{i-1}}^s \nabla g_j(X_u)\nu(X_u) \, dW_u\,\bigg| \, \mu(X_{i-1}, \theta_0)\right>\right)^2\,ds\right)\leq\\
&\leq \frac{1}{\Delta_n}\sum_{i=1}^n \mathbb{E} \left(\int_{t_{i-1}}^{t_i} \norm{\int_{t_{i-1}}^s \nabla g_j(X_u)\nu(X_u)\, dW_u}_2^2 \cdot \norm{\mu(X_{i-1}, \theta_0)}_2^2\, ds\right)\leq\\
&\leq \frac{c_3}{\Delta_n} \sum_{i=1}^n \Delta_n \int_{t_{i-1}}^{t_i} \mathbb{E}\left(\norm{\int_{t_{i-1}}^s \nabla g_j(X_u)\nu(X_u)\, dW_u}_2^2\right)=\\
&=\frac{c_3}{\Delta_n}\sum_{i=1}^n \Delta_n \int_{t_{i-1}}^{t_i} \mathbb{E}\left(\int_{t_{i-1}}^s \norm{\nabla g_j(X_u)\nu(X_u)}_F^2\,du\right)\leq\\
&\leq c_4 \Delta_n T
\end{align*}
Using Doob's maximal inequality for vector martingale \cite[Theorem 1.7]{ry} and Lemma \ref{iso} \eqref{miito} we conclude that there exists constant $c_{5}$ such that  \eqref{a12b} + \eqref{a12c} is bounded in the following way
\begin{align*}
&\Bigg|\frac{1}{\sqrt{\Delta_n}} \sum_{i=1}^n \int_{t_{i-1}}^{t_i} \left<\int_{t_{i-1}}^s \nabla g_j(X_u) \nu(X_u) \, dW_u\,\bigg|\, \int_{t_{i-1}}^{s} \nabla \mu(X_u, \theta_0)\mu(X_u, \theta_0)\, du\right>ds\\
&+ \frac{1}{\sqrt{\Delta_n}} \sum_{i=1}^n \int_{t_{i-1}}^{t_i} \left<\int_{t_{i-1}}^s \nabla g_j(X_u) \nu(X_u) \, dW_u\,\bigg| \, \frac{1}{2} \int_{t_{i-1}}^s
\nabla_2 \mu(X_u, \theta_0)\,du\right>\,ds\Bigg|\leq \\
&\leq  c_{5} \sqrt{\Delta_n}T.
\end{align*}
To prove convergence of \eqref{a12d} we use It\^o formula for the function $F: \mathbb{R}^k \times \mathbb{R}^k \rightarrow \mathbb{R}$, $F(y,z)=\left<y\,|\,z\right>$ on $[t_{i-1}, s]$ and vector martingales
\begin{align*}
Y_s \coloneqq \int_{t_{i-1}}^s \nabla g_j(X_u)\nu(X_u) \,dW_u, \quad
Z_s \coloneqq \int_{t_{i-1}}^s \nabla \mu(X_u, \theta_0)\nu(X_u) \,dW_u.
\end{align*}
When we calculate quadratic variation $\left<Y^l, Z^l\right>_s$, it is crucial to use independence of components of Brownian motion.
Then, \eqref{a12d} equals
\begin{align}
&\frac{1}{\sqrt{\Delta_n}} \sum_{i=1}^n \int_{t_{i-1}}^{t_i} \left<\int_{t_{i-1}}^s \nabla g_j(X_u) \nu(X_u)  dW_u\bigg|  \int_{t_{i-1}}^{s} \nabla \mu(X_u, \theta_0) \nu(X_u)  dW_u \right>ds=\notag\\
&+\frac{1}{\sqrt{\Delta_n}} \sum_{i=1}^n \int_{t_{i-1}}^{t_i} \int_{t_{i-1}}^s \left<\int_{t_{i-1}}^u \nabla \mu (X_v, \theta_0)\nu(X_v) \,dW_v\bigg|  \nabla g_j(X_u)\nu(X_u)\,dW_u\right>ds+\label{a12d1}\\
&+\frac{1}{\sqrt{\Delta_n}} \sum_{i=1}^n \int_{t_{i-1}}^{t_i} \int_{t_{i-1}}^s \left<\int_{t_{i-1}}^u \nabla g_j(X_v)\nu(X_v) \, dW_v\bigg| \nabla\mu(X_u, \theta_0) \nu(X_u) \, dW_u\right>ds +\label{a12d2}\\
&+\frac{1}{\sqrt{\Delta_n}} \sum_{i=1}^n \int_{t_{i-1}}^{t_i} \int_{t_{i-1}}^s \sum_{p,r=1}^k \left(\left(\nabla g_j(X_u)\nu(X_u)\right) \circ \left(\nabla \mu(X_u, \theta_0)\nu(X_u)\right)\right)_{pr} \,du\,ds\label{a12d3}.
\end{align}
Expressions \eqref{a22}, \eqref{a12d1} and \eqref{a12d2} we treat in the same manner as \eqref{a12a}. Because of boundedness of function in \eqref{a12d3}, it is bounded by $c_{6} \sqrt{\Delta_n}T$.\\
By a similar reasoning as for \eqref{a12a} there exists constant $c_{7}$ such that
\begin{align*}
\mathbb{E}\left[\left( \frac{1}{\sqrt{\Delta_n}}\sum_{i=1}^n \int_{t_{i-1}}^{t_i} \left< \int_{t_{i-1}}^s G_j(X_u, \theta_0) \, du\,\bigg| \, \nu(X_s)\, dW_s \right> \right)^2\right]\leq c_{7} \Delta_n T.
\end{align*}
Hence, we can conclude that  \eqref{a31} converges in $L^2$ to zero so it converges also in probability to zero.\\
In the end, every component of $V_n(\theta)$ converges in probability to zero so we can conclude that the whole vector $V_n(\theta) \overset{\mathbb{P}}{\to} \mathbf{0}_d$, $n \to +\infty$.
\end{proof}
\begin{proof} [Proof of Lemma \ref{pl}]
\begin{align*}
\mathbb{E}\left[\norm{X_{s_2}-X_{s_1}}_2^2\right]\leq 2\mathbb{E}\left[\norm{\int_{s_1}^{s_2}\mu(X_t, \theta_0)\,dt}_2^2\right]+2\mathbb{E}\left[\norm{\int_{s_1}^{s_2} \nu(X_t)\, dW_t}_2^2\right]
\end{align*}
For the first integral we use Cauchy-Schwarz inequality for integrals, and for the second one we use Lemma \ref{iso} \eqref{miito}. Finally, we have
\begin{align*}
\mathbb{E}\left[\norm{X_{s_2}-X_{s_1}}_2^2\right]&\leq 2(s_2-s_1)\mathbb{E}\left[\int_{s_1}^{s_2}\norm{\mu(X_t, \theta_0)}_2^2\,dt\right]+\\
&+2 \mathbb{E}\left[\int_{s_1}^{s_2} \norm{\nu(X_t)}_F^2\,dt\right]\leq\\
&\leq 2K \left(\left(s_2-s_1\right)^2+\left(s_2-s_1\right)\right).
\end{align*}
\end{proof}

\begin{proof} [Proof of Lemma \ref{fk}]
Without loss of generality (Remark \ref{four}) we can assume that assumption (P2) is fulfilled for function $f$.
We prove \eqref{fk1}. When $m=k_1+\dots+k_d$ and $\varepsilon_j=\text{sign}(k_j)$, $j=1,\dots,d$, we have
\begin{align}\label{ras1}
&\left(1+ |k_1|+|k_2|+\cdots+|k_d|\right)^m \norm{C_{\mathbf{k}}(x)-C_{\mathbf{k}}(y)}_2=\notag\\
&=\left(1+\varepsilon_1k_1+\cdots+\varepsilon_d k_d\right)^m \norm{C_{\mathbf{k}}(x)-C_{\mathbf{k}}(y)}_2=\notag\\
&=\sum_{j_0+\cdots+j_d=m} \binom{m}{j_0,j_1,\dots, j_d} \varepsilon_1^{j_1}\cdots \varepsilon_d^{j_d} \mathbf{k}^{\mathbf{j}} \norm{C_{\mathbf{k}}(x)-C_{\mathbf{k}}(y)}_2 \leq\notag\\
&\leq\left(d+1\right)^m \norm{C_{\mathbf{k}}^{(\mathbf{j})}(x)-C_{\mathbf{k}}^{(\mathbf{j})}(y)}_2.
\end{align}
Here we use the multinomial theorem and the property of Fourier coefficients, namely $C_{\mathbf{k}}^{(\mathbf{j})}(x)=i^m \mathbf{k}^{\mathbf{j}}C_{\mathbf{k}}(x)$. Using the definition of Fourier coefficient and Lemma \ref{mvt}, inequality \eqref{ras1} yields
\begin{align*}
&\norm{C_{\mathbf{k}}(x)-C_{\mathbf{k}}(y)}_2 \leq \\
&\leq K_m(\mathbf{k}) \norm{C_{\mathbf{k}}^{(\mathbf{j})}(x)-C_{\mathbf{k}}^{(\mathbf{j})}(y)}_2 \leq\\
&\leq K_m(\mathbf{k}) \frac{1}{(2\pi)^d} \int_{Cl(\mathcal{K}_0)} \norm{D_{\mathbf{j}}^mf(x, \theta)-D_{\mathbf{j}}^mf(y, \theta)}_2\,d\theta\leq\\
&\leq K_m(\mathbf{k}) \frac{1}{(2\pi)^d} \int_{Cl(\mathcal{K}_0)} \int_0^1 \norm{\nabla_x D_{\mathbf{j}}^mf\left(x+s(y-x),\theta\right)}_F\,ds\cdot\norm{x-y}_2\,d\theta\leq\\
&\leq k_1 K_m(\mathbf{k}) \norm{x-y}_2.
\end{align*}
The beginning of proving \eqref{fk2} is the same as in \eqref{ras1}. Afterwards we analyze it in the following way.
\begin{align*}
\norm{C_{\mathbf{k}}(x)}_2
&\leq \left(\frac{d+1}{1+|k_1|+\cdots+|k_d|}\right)^m \frac{1}{(2\pi)^d} \int_{Cl(\mathcal{K}_0)}\norm{D_{\mathbf{j}}^m f(x,\theta)}_2\, d\theta\leq\\
&\leq k_2 K_m(\mathbf{k})
\end{align*}
\end{proof}

\begin{proof} [Proof of Lemma \ref{fk3}]
The proof is the same as one part of proof of Lemma 6.7 in \cite{huzak2018}.
\end{proof}


\end{document}